\documentclass[11pt,reqno]{amsart}
\usepackage[margin=1in]{geometry}

\usepackage{graphicx}
\usepackage{amsmath,amsthm,amssymb}
\usepackage[mathscr]{eucal}
\usepackage{enumerate,url}
\usepackage{multicol}
\usepackage{color}
\usepackage{tikz-cd}
\usepackage{mathtools}

\newtheorem{theorem}{Theorem}[section]

\newtheorem{proposition}[theorem]{Proposition}
\newtheorem{proposition/definition}[theorem]{Proposition/Definition}
\newtheorem{lemma}[theorem]{Lemma}
\newtheorem{corollary}[theorem]{Corollary}

\theoremstyle{definition}
\newtheorem{definition}[theorem]{Definition}
\newtheorem{definition/proposition}[theorem]{Definition/Proposition}

\newtheorem{example}[theorem]{Example}

\newtheorem{construction}[theorem]{Construction}

\theoremstyle{remark}

\DeclareMathOperator{\id}{id}

\DeclareMathOperator{\diag}{diag}
\DeclareMathOperator{\Hank}{Hank}
\DeclareMathOperator{\Circ}{Circ}
\DeclareMathOperator{\Toep}{Toep}
\DeclareMathOperator{\proj}{proj}
\DeclareMathOperator{\bttb}{\textsc{bttb}}

\DeclareMathOperator{\rank}{rank}
\DeclareMathOperator{\brank}{\overline{\rank}}

\definecolor{darkblue}{rgb}{0,0,0.7}
\newcommand{\change}[1]{#1}

\title{Fast structured matrix computations: tensor rank and Cohn--Umans method}
\author[K.~Ye]{Ke~Ye}
\address{Department of Mathematics and Department of Statistics, University of Chicago, Chicago, IL 60637-1514.}
\email{kye@math.uchicago.edu}
\author[L.-H.~Lim]{Lek-Heng~Lim}
\address{Computational and Applied Mathematics Initiative, Department of Statistics,
University of Chicago, Chicago, IL 60637-1514.}
\email[corresponding author]{lekheng@galton.uchicago.edu}
\begin{document}

\begin{abstract}
We discuss a generalization of the Cohn--Umans method, a potent technique developed for studying the bilinear complexity of matrix multiplication by embedding matrices into an appropriate group algebra. We investigate how the Cohn--Umans method  may be used for bilinear operations other than matrix multiplication, with algebras other than group algebras, and we relate it to Strassen's tensor rank approach, the traditional framework for investigating bilinear complexity. To demonstrate the utility of the generalized method, we apply it to find the fastest algorithms for forming structured matrix-vector product, the basic operation underlying iterative algorithms for structured matrices. The structures we study include Toeplitz, Hankel, circulant, symmetric, skew-symmetric,  $f$-circulant, block-Toeplitz-Toeplitz-block, triangular Toeplitz matrices, Toeplitz-plus-Hankel, sparse/banded/triangular. Except for the case of skew-symmetric matrices, for which we have only upper bounds, the algorithms derived using the generalized Cohn--Umans method in all other instances are the fastest possible in the sense of having minimum bilinear complexity. We also apply this framework to a few other bilinear operations including matrix-matrix, commutator, simultaneous matrix products, \change{and briefly discuss the relation between tensor nuclear norm and numerical stability.}
\end{abstract}

\maketitle

\section{Introduction}

In this article, we systematically study the design of fast, possibly fastest, algorithms for a variety of operations involving structured matrices, as measured by the \textit{bilinear complexity} of the problem. \change{Roughly speaking, the bilinear complexity of an algorithm for a problem that can be cast as the evaluation of a bilinear map is the number of multiplications required in the algorithm; the bilinear complexity of the problem is then that of an algorithm with the lowest bilinear complexity \cite[Chapter~14]{BCS}.} This notion of complexity is best known for its use in quantifying the speed of matrix-matrix product and matrix inversion in the work of Strassen \cite{Strass0}, Coppersmith--Winograd \cite{CW}, Vassilevska~Williams \cite{Williams}, and many others. The current record,  due to Le Gall \cite{LeGall}, for the asymptotic bilinear complexity of $n \times n$ matrix-matrix product for unstructured matrices is $O(n^{2.3728639})$. \change{Roughly speaking, the asymptotic bilinear complexity of a problem dependent on $n$ refers to its bilinear complexity when $n$ is sufficiently large.}

The algorithms that we study in article will be for the following operations: (1) matrix-vector product,  (2) matrix-matrix product, and (3) commutator product:
\[
(A, x) \mapsto Ax, \qquad (A,B) \mapsto AB, \qquad (A,B) \mapsto AB - BA,
\]
where $A$ and $B$ are structured matrices and $x$ is a vector, of appropriate dimensions so that the products are defined.

The structured matrices studied in this article include: (i) sparse (including banded and triangular), (ii) symmetric, (iii) skew-symmetric, (iv) Toeplitz, (v) Hankel, (vi) circulant, (vii) $f$-circulant and skew-circulant, (viii) block-Toeplitz-Toeplitz-block (\textsc{bttb}) and more generally any block structured matrices with structured blocks, (ix) triangular Toeplitz and its analogues for Hankel and circulant matrices, (x) sum of Toeplitz and Hankel. We provide algorithms of optimal bilinear complexity for all except the skew-symmetric case (for which we only have upper bounds). The optimal bilinear complexity for the Toeplitz and triangular Toeplitz matrix-vector product are well-known, due to Bini and Capovani \cite{Bini/Capovani}, but we will obtain them using a different method (generalized Cohn--Umans) that applies more generally to all classes of structured matrices discussed here.

We will examine two different approaches: the Strassen tensor rank approach \cite{Strass1,Strass2}, and the  Cohn--Umans group theoretic approach \cite{CU1, CKSU, CU2}, as well as the relations between them. Our study gives a generalization of the Cohn--Umans approach in two regards: a generalization from matrix-matrix product to arbitrary bilinear operations, and a generalization from (a) group algebras (e.g., Section~\ref{sec:simul}) to arbitrary algebras including (b) cohomology rings of manifolds (e.g., Section~\ref{sec:triangular}), (c) coordinate rings of schemes (e.g., Section~\ref{sec:triangular}) and varieties (e.g., Section~\ref{sec:bttb}), (d) polynomial identity rings (e.g., Section~\ref{sec:comm}). We will provide the equivalent of their `triple product property' in these more general contexts. \change{The idea of considering algebras other than group algebras was already in \cite{CU2}, where the authors proposed to use  adjacency algebras of coherent configurations. These may be viewed as a generalization of group algebras and are in particular semisimple, i.e., isomorphic to an algebra of block diagonal matrices. Our generalization goes further in that the algebras we use may contain nilpotents and thus cannot be semisimple (e.g., Section~\ref{sec:triangular}); in fact they may not be associative algebras (e.g., Section~\ref{sec:comm}), may not be algebras (e.g., Section~\ref{sec:skew-symmetric}), and may not even be vector spaces (e.g., Example~\ref{eg:intmult}).}

We hope to convince our readers, by way of a series of constructions involving various structured matrices and various bilinear operations, that this generalization of Cohn--Umans method could allow one to systematically uncover fast algorithms, and these could in turn be shown to be the fastest possible  (in terms of bilinear complexity) via arguments based on the Strassen tensor rank approach. For instance, we will see in Section~\ref{sec:symm} that the fastest possible algorithm for multiplying a symmetric matrix to a vector involves first writing the symmetric matrix as a sum of Hankel matrices of decreasing dimensions bordered by zeros. For example, a $4 \times 4$ symmetric matrix would have to be decomposed into
\[
\begin{bmatrix}
a & b & c & d\\
b & e & f & g\\
c & f & h & i \\
d & g & i & j
\end{bmatrix}
=
\begin{bmatrix}
a & b & c & d\\
b & c & d & g\\
c & d & g & i \\
d & g & i & j
\end{bmatrix}
+
\begin{bmatrix}
0 & 0 & 0 & 0\\
0 & e - c & f -d & 0\\
0 & f  -d & e -c & 0 \\
0 & 0 & 0 & 0
\end{bmatrix}
+
\begin{bmatrix}
0 & 0 & 0 & 0\\
0 & 0 & 0 & 0\\
0 & 0 & h -g -e + c & 0 \\
0 & 0 & 0 & 0
\end{bmatrix}.
\]
This is highly nonobvious to us. We would not have been able to find this algorithm without employing the generalized Cohn--Umans approach.

The main focus of our article will be the \textit{matrix-vector product} for various structured matrices since these form the fundamental building blocks of most modern iterative algorithms for problems involving structured matrices: linear systems \cite{Chan, Ng}, least-squares problems \cite{Bjorck,Ng}, eigenvalue problems \cite{Watkins}, evaluating analytic functions with matrix arguments \cite{Higham}, etc. On the other hand, problems requiring matrix-matrix and product of structured matrices are relatively uncommon; one reason being that the most common structured matrices (symmetric, Toeplitz, Hankel, etc; in fact all but circulant) are not closed under matrix-matrix products. \change{Explicit pseudocodes for all structured matrix-vector product algorithms appearing in this article may be found in \cite{KL2}.}

\subsection{Why minimize multiplications?}

In modern computer processors, there is no noticeable difference in the latency of addition and multiplication \cite[Tables~14-1 and 15-6]{Intel}. So the reader might wonder why bilinear complexity continues to be of relevance. We provide three reasons below.

The first reason is that such algorithms apply when we have matrices in place of scalars. We illustrate this with a simple example, Gauss's method for multiplying two complex numbers \cite[Section~4.6.4]{Knuth}. Let $a,b,c,d \in \mathbb{R}$. Then the usual method
\[
(a + ib)(c+ id) = (ac - bd) + i (ad +bc)
\]
requires four real multiplications and two real additions but Gauss's method
\begin{equation}\label{eq:gauss}
(a + ib)(c+ id) = (ac - bd) + i [(a+b)(c+d) -ac -bd] 
\end{equation}
requires three real mutliplications and five real additions. If the costs of addition and multiplication are roughly the same, then Gauss's method is a poor way for multiplying complex numbers. However, the usefulness of Gauss's method comes into view when we multiply complex \textit{matrices} \change{\cite[Chapter~23]{Higham1}}, i.e., when we do
\[
(A + iB)(C+ iD) = (AC - BD) + i[(A+B)(C+D) -AC -BD]
\]
where $A,B,C,D \in \mathbb{R}^{n \times n}$. Now Gauss's method requires three \textit{matrix multiplications} instead of four. Addition and multiplication of scalars may well have similar computational costs but multiplication of $n \times n$  matrices is by any measure vastly more expensive\footnote{Even if the exponent of matrix multiplication turns out to be $2$; note that this is asymptotic.} than addition of $n \times n$ matrices. This observation applies more generally. For example, Strassen's algorithm for the product of $2 \times 2$ matrices \cite{Strass0} only becomes practically useful when it is applied (recursively) to the product of $2 \times 2$ \textit{block matrices} \change{\cite[Chapter~23]{Higham1}}.

A second reason is that \change{the preceding comparison of addition and multiplication implicitly assumes that we are using} the traditional measure of computational cost, i.e., time complexity, but other measures, e.g., energy consumption, number of gates, code space, etc, have become increasingly important. For instance, a multiplier requires many more gates than an  adder (e.g., 2200 gates for an 18-bit multiplier versus 125 gates for an 18-bit adder \cite{Kodavalla}), which translates into more wires and transistors on a microchip and also consumes more energy.

A third reason is that while the latencies of addition and multiplication are comparable on a general purpose \textsc{cpu}, it is important to remember that arithmetic is performed on other microchips as well, e.g., \textsc{asic}, \textsc{dsp}, \textsc{fpga}, \textsc{gpu}, motion coprocessor, etc, where \change{the latency of multiplication may be substantially higher than that of addition}. Moreover, our second reason also applies in this context.

\subsection{\change{Overview}}\label{sec:overview}  We begin by introducing the central object of this article, the structure tensor of a bilinear operation, and discuss several examples in Section~\ref{sec:struct}. This is followed by a discussion of tensor rank  and the closely related notion of border rank in Section~\ref{sec:rank}, allowing us to define bilinear complexity rigorously as the rank of a structure tensor. We proved several results regarding tensor rank and border rank that will be useful later when we need to determine these for a given structure tensor. We end the section with a brief discussion of numerical stability and its relation to the nuclear norm of the structure tensor. 

In Section~\ref{sec:rankstruct}, we examine the structure tensor in the special case where the bilinear operation is the product operation in an algebra and prove a relation between tensor ranks of the respective structure tensors when one algebra is mapped into another. This provides partial motivation for the generalized Cohn--Umans method in Section~\ref{sec:cohn}, where we first present the usual Cohn--Umans method as a commutative diagram of algebras and \emph{vector space homomorphisms} (as opposed to homomorphisms of algebras), followed by a demonstration that the `triple product property' is equivalent to the commutativity of the diagram. Once presented in this manner, the Cohn--Umans method essentially generalizes itself. As a first example, we show that the fast integer multiplication algorithms of Karatsuba et al.\ may be viewed as an application of the generalized Cohn--Umans method.

In the remainder of the article, we apply the generalized Cohn--Umans method to analyze a variety of structured matrix-vector products:
\begin{multicols}{2}
\begin{itemize}
\item sparse, banded, triangular: Section~\ref{sec:sparse},
\item circulant: Section~\ref{sec:circulant},
\item $f$-circulant, skew-circulant: Section~\ref{sec:f-circ},
\item Toeplitz: Section~\ref{sec:toep},
\item Hankel: Section~\ref{sec:hank},
\item triangular Toeplitz/Hankel: Section~\ref{sec:triangular},
\item Toeplitz-plus-Hankel: Section~\ref{sec:t+h},
\item block-Toeplitz-Toeplitz-block and other multilevel structures: Section~\ref{sec:bttb},
\item symmetric: Section~\ref{sec:symm},
\item skew-symmetric: Section~\ref{sec:skew-symmetric}.
\end{itemize}
\end{multicols}
Aside from the case of skew-symmetric matrices, we obtain algorithms with optimum bilinear complexities for all structured matrix-vector products listed above. In particular we obtain the rank and border rank of the structure tensors in all cases but the last.

A reader who follows the developments in Sections~\ref{sec:circulant}--\ref{sec:skew-symmetric} will observe a certain degree of interdependence between these algorithms. For example, as we have mentioned earlier, the algorithm for symmetric matrix-vector product depends on that for Hankel matrix-vector product, but the latter depends on that for Toeplitz matrix-vector product, which in turn depends on that for circulant matrix-vector product. As another example of a somewhat surprising interdependence, in Section~\ref{sec:comm}, we discuss an algorithm for the commutator product, i.e., $[A,B] = AB -BA$, for $2 \times 2$ matrices $A, B$ based on the algorithm for $3 \times 3$ skew-symmetric matrix-vector product in Section~\ref{sec:skew-symmetric}. Yet a third example is that our algorithm for skew-circulant matrix-vector product in Section~\ref{sec:f-circ} turns out to contain Gauss's multiplication of complex numbers as a special case: \eqref{eq:gauss} may be viewed as the product of a  skew-circulant matrix in $\mathbb{R}^{2 \times 2}$ with a vector in $\mathbb{R}^2$.

To round out this article, we introduce a new class of problems in Section~\ref{sec:simul} that we call `simultaneous product' of matrices. The most natural problem in this class would be the simultaneous computation of $AB$ and $AB^\mathsf{T}$ for a square matrix $B$ but we are unable to obtain any significant findings in this case. Nevertheless we provide an impetus by showing that the closely related variants of  simultaneously computing the pair of matrix products
\[
\begin{bmatrix}
a & b\\
c & d
\end{bmatrix} \begin{bmatrix}
e & f\\
g & h
\end{bmatrix}
\qquad \text{and} \qquad \begin{bmatrix}
a & b \\
c & d
\end{bmatrix}\begin{bmatrix}
g & h \\
e & f
\end{bmatrix},
\]
or  the pair of matrix products
\[
\begin{bmatrix}
a & b\\
c & d
\end{bmatrix} \begin{bmatrix}
e & f\\
g & h
\end{bmatrix}
\qquad \text{and} \qquad \begin{bmatrix}
a & b \\
c & d
\end{bmatrix}\begin{bmatrix}
h & g \\
e & f
\end{bmatrix},
\]
can be obtained with just \emph{eight} multiplications and that the resulting algorithms have optimum bilinear complexity. Note that computing the pair of products separately via Strassen's algorithm, which is optimum for $2 \times 2$ matrix-matrix product, would require $14$ multiplications.

Throughout this article, we work over $\mathbb{C}$ for simplicity but our results hold for more general fields --- quadratic, cyclotomic, infinite, or algebraically closed extensions of an arbitrary field (say, a finite field), depending on the context.

Results in Sections~\ref{sec:struct}--\ref{sec:sparse} and Section~\ref{sec:simul} are independent of our choice of field with a few exceptions: (i) any discussion of Gauss's method is of course peculiar to $\mathbb{C}$ but generalizes to any quadratic extension of an arbitrary field; (ii) the discussion of numerical stability in Section~\ref{sec:stability} require that we work over a subfield of $\mathbb{C}$ since they involve norms; (iii) Winograd's theorem (Theorem~\ref{thm:Winograd}) requires an infinite field; (iv) Corollary~\ref{cor:realization of any bilinear map over alg. closed field} requires an algebraically closed field. The results in Sections~\ref{sec:circulant}--\ref{sec:symm} for $n \times n$ structured matrices require that the field contains all $n$th roots of some element, usually $1$ but sometimes $-1$ (for skew-circulant or skew-symmetric) or $f$ (for $f$-circulant). Results in Sections~\ref{sec:skew-symmetric} and \ref{sec:comm} require an algebraically closed field.

\section{The structure tensor of a bilinear operation}\label{sec:struct}

A \textit{bilinear operation} is simply a bilinear map $\beta : U \times V \to W$ where $U,V,W$ are vector spaces over the same field, henceforth assumed to be $\mathbb{C}$. For example, the operation of forming a matrix-vector product is a bilinear operation $\beta : \mathbb{C}^{m \times n} \times \mathbb{C}^n \to \mathbb{C}^m$, $(A, x) \mapsto Ax$, since
\[
\beta(a A + bB, x) = a\beta(A,x) + b \beta(B,x), \qquad \beta(A, ax+ by) = a \beta(A,x) + b \beta(A,y).
\]
Likewise for the operations of matrix-matrix product and commutator product.

A simple but central observation in the study of bilinear complexity is that every bilinear operation is characterized by a $3$-tensor and that its tensor rank quantifies the complexity, as measured solely in terms of the number of multiplications, of the bilinear operation. We start by defining this $3$-tensor.
\begin{definition/proposition}\label{prop:bilinear-linear}
Let $\beta:U\times V\to W $ be a bilinear map. Then there exists a unique tensor $\mu_\beta \in U^* \otimes V^* \otimes W$ such that given any $(u,v)\in U\times V$ we have 
\[
\beta(u,v)=\mu_\beta(u,v,\cdot)\in W.
\]
We call $\mu_\beta$ the \textit{structure tensor} of the biliear map $\beta$.
\end{definition/proposition}
By the definition of tensor product, there is a one-to-one correspondence between the set of bilinear maps from $U\times V$ to $W$ and the set of linear maps from $U\otimes V$ to $W$. Therefore we do not distinguish between a bilinear map $\beta : U \times V \to W$ and its corresponding linear map $\beta : U \otimes V \to W$ (and denote both by $\beta$).

In the special case when $U=V=W=\mathcal{A}$ is an algebra and the bilinear map  $\beta:\mathcal{A}\times \mathcal{A}\to \mathcal{A}$, $(u,v) \mapsto uv$, is multiplication in $\mathcal{A}$. The structure tensor of $\beta$ is called the \textit{structure tensor of the algebra} $\mathcal{A}$, and is denoted by $\mu_{\mathcal{A}}$.

\begin{example}[Lie algebras] Let $\mathfrak{g}$ be a complex Lie algebra of dimension $n$ and let $\{e_1,  \dots, e_n\}$ be a basis of $\mathfrak{g}$. Let $\{e^*_1,  \dots, e^*_n\}$ be the corresponding dual basis defined in the usual way as
\[
e^*_i(e_j) =
\begin{cases}
1 & i = j,\\
0 & i \ne j,
\end{cases}
\]
for all $i,j=1,\dots,n$.
Then for each pair $i,j\in \{1,\dots, n\}$,
\[
\left[e_i,e_j\right]=\sum_{k=1}^n c_{ij}^k e_k,
\]
for some constant numbers $c_{ij}^k\in \mathbb{C}$. The structure tensor of the Lie algebra $\mathfrak{g}$ is 
\[
\mu_{\mathfrak{g}}=\sum_{i,j,k=1}^n c_{i,j}^k e^*_i\otimes e^*_j \otimes e_k\in \mathfrak{g}^* \otimes \mathfrak{g}^* \otimes \mathfrak{g}.
\]
The constants $c_{ij}^k$ are often called the \textit{structure constants} of the Lie algebra and the hypermatrix  \cite{HLA} 
\[
[c_{ij}^k] \in \mathbb{C}^{n \times n \times n}
\]
is the coordinate representation of $\mu_{\mathcal{A}}$ with respect to the basis $\{e_1,  \dots, e_n\}$.

For a specific example, take $\mathfrak{g}=\mathfrak{so}_3$, the Lie algebra of real $3\times 3$ skew-symmetric matrices and consider the basis of $\mathfrak{g}$ comprising
\[
e_1=\begin{bmatrix}
0 & 0 & 0\\
0 & 0 & -1\\
0 & 1 & 0
\end{bmatrix},\qquad
e_2=\begin{bmatrix}
0 & 0 & -1\\
0 & 0 & 0\\
1 & 0 & 0
\end{bmatrix},\qquad
e_3=\begin{bmatrix}
0 & -1 & 0\\
1 & 0 & 0\\
0 & 0 & 0
\end{bmatrix},
\]
with  dual $e^*_1,e^*_2,e^*_3$. Then the structure tensor of $\mathfrak{so}_3$ is 
\[
\mu_{\mathfrak{so}_3}=\sum_{i,j,k=1}^3 \epsilon_{ij}^k e^*_i \otimes e^*_j \otimes e_k,
\]
where 
\[
\epsilon_{ij}^k=\frac{(i-j)(j-k)(k-i)}{2},
\]
often called the Levi-Civita symbol.
\end{example}

\begin{example}[Matrix multiplication] Consider the usual matrix product $\beta : \mathbb{C}^{m \times n} \times \mathbb{C}^{n \times p} \to \mathbb{C}^{m \times p}$, $(A,B) \mapsto AB$. We let $E_{ij}$ be the elementary matrix with one in the $(i,j)$th entry and zeros elsewhere and $E_{ij}^*$ be the dual. Then the structure tensor for this bilinear operation is 
\[
\mu_{m,n,p}=\sum_{i,j,k=1}^{m,n,p} E_{ij}^* \otimes E_{jk}^*\otimes E_{ik},
\]
the famous \textit{Strassen matrix multiplication tensor}. With respect to these bases, $\mu_{m,n,p}$ is an $mn \times np \times mp$ hypermatrix whose entries are all zeros and ones. When $m = n = p$, this becomes the structure tensor of the matrix algebra $\mathbb{C}^{n \times n}$.
\end{example}

\begin{example}[Matrix-vector multiplication]\label{eg:matvec} Consider the bilinear map
\[
\beta:\mathbb{C}^{n \times n}\times \mathbb{C}^n \to \mathbb{C}^n, \quad (A,x) \mapsto Ax.
\]
Let $\{E_{ij}\in \mathbb{C}^{n \times n}: i,j =1,\dots,n\}$ and $\{e_i \in \mathbb{C}^n: i =1,\dots,n\}$ be the standard bases for $\mathbb{C}^{n\times n}$ and $\mathbb{C}^{n}$ respectively. Then the structure tensor of $\beta$ is 
\[
\mu_\beta=\sum_{i,j=1}^n E^*_{ij}\otimes e^*_{j}\otimes e_i.
\]
With respect to these bases, $\mu_\beta$ is an $n^2 \times n \times n$ hypermatrix whose entries are all zeros and ones. This is of course nothing more than a special case of the previous example with $m = n$ and $p = 1$.
\end{example}

A comment is in order for those who are not familiar with multilinear algebra and wonders about the difference between $\beta$ and $\mu_\beta$. Given a bilinear map $\beta:U\times V\to W$ there exists a unique trilinear function $\tilde{\beta}:U\times V\times W^*\to \mathbb{C}$ such that given any $(u,v,\omega)\in U\times V\times W^*$ we have
\[
\omega(\beta(u,v))=\tilde{\beta}(u,v,\omega).
\]
Furthermore, both $\beta$ and $\tilde{\beta}$ correspond to the same tensor $\mu_\beta\in U^*\otimes V^*\otimes W$ and so $\mu_\beta$ quantifies both the bilinear operation $\beta$ and the trilinear operation $\tilde{\beta}$. As a concrete example, consider Example~\ref{eg:matvec} where $\beta:\mathbb{C}^{n \times n}\times \mathbb{C}^n \to \mathbb{C}^n$ is the matrix-vector product
\[
\beta(A,x)=Ax.
\]
Then $\tilde{\beta}:\mathbb{C}^{n \times n}\times \mathbb{C}^n\times \mathbb{C}^n\to \mathbb{C}$ is 
\[
\tilde{\beta}(A,x,y)=y^\mathsf{T} A x,
\]
and they correspond to the same tensor $\mu_\beta\in (\mathbb{C}^{n \times n})^*\otimes (\mathbb{C}^n) ^*\otimes \mathbb{C}^n$.

We conclude this section with the simplest example, but worked out in full details for the benefit of readers unfamiliar with multilinear algebra.
\begin{example}[Complex number multiplication]\label{eg:cplx}
Complex numbers form a two-dimensional algebra over $\mathbb{R}$ and the multiplication of complex numbers is an $\mathbb{R}$-bilinear map 
\[
\beta:\mathbb{C}\times \mathbb{C}\to \mathbb{C}, \quad (a + bi, c+di) \mapsto (ac -bd) + (ad + bc)i,
\]
for any $a,b,c,d \in \mathbb{R}$.
Let $e_1= 1 + 0i = 1$ and  $e_2= 0 + 1i = i$ be the standard basis of $\mathbb{C}$ over $\mathbb{R}$ and let $e^*_1, e^*_2$ be the corresponding dual basis. The structure tensor of $\mathbb{C}$ is, by definition, the structure tensor of $\beta$ and is given by
\begin{equation}\label{eq:cplx1}
\mu_{\mathbb{C}} = \mu_\beta =e^*_1\otimes e^*_1\otimes e_1  -e^*_2\otimes e^*_2\otimes e_1 + e^*_1\otimes e^*_2\otimes e_2+e^*_2 \otimes e^*_1\otimes e_2,
\end{equation}
or, as a hypermatrix with respect to these bases,
\begin{equation}\label{eq:cplx}
\mu_{\mathbb{C}} =
\biggl[ \begin{matrix}
1 & 0\\
0 & -1
\end{matrix} \biggm|
\begin{matrix}
0 & 1\\
1 & 0
\end{matrix} \biggr] \in \mathbb{R}^{2 \times 2 \times 2}.
\end{equation}
We provide here a step-by-step verification that $\mu_{\mathbb{C}}$ is indeed the structure tensor for complex number multiplication over $\mathbb{R}$. Given two complex numbers $z_1 = a+bi$, $z_2=c+di \in \mathbb{C}$, we write them as $z_1 =ae_1+be_2$, $z_2=ce_1+de_2$. Then
\begin{align*}
\mu_{\mathbb{C}}(z_1,z_2) &= [(e^*_1(z_1) e^*_1(z_2) -e^*_2(z_1) e^*_2(z_2)] e_1 + [e^*_1(z_1) e^*_2(z_2)+e^*_2(z_1) e^*_1(z_2)] e_2\\
&= [(e^*_1(ae_1+be_2) e^*_1(ce_1+de_2) -e^*_2(ae_1+be_2) e^*_2(ce_1+de_2)] e_1 \\
&\qquad \qquad+ [e^*_1(ae_1+be_2) e^*_2(ce_1+de_2)+e^*_2(ae_1+be_2) e^*_1(ce_1+de_2)] e_2\\
&=(ac-bd)e_1 + (ad+bc) e_2=(ac-bd,ad+bc)\\
&=(ac-bd)+(ad+bc)i.
\end{align*}
For the uninitiated wondering the usefulness of all these, we will see in Section~\ref{sec:rank} that the notion of tensor rank and its associated rank decomposition allow us to discover faster, possibly fastest, algorithms for various bilinear operations. For instance, the four-term decomposition in \eqref{eq:cplx} gives us the usual algorithm for multiplying complex numbers but as we will see, one may in fact obtain a three-term decomposition for $\mu_{\mathbb{C}}$,
\begin{equation}\label{eq:cplx2}
\mu_{\mathbb{C}}=(e^*_1+e^*_2)\otimes (e^*_1+e^*_2)\otimes e_2 + e^*_1\otimes e^*_1\otimes (e_1-e_2) - e^*_2\otimes e^*_2\otimes (e_1+e_2).
\end{equation}
This gives us Gauss's method for multiplying two complex numbers with three real multiplications that we saw in \eqref{eq:gauss}.
We will have more to say about Gauss's method in Section~\ref{sec:f-circ} --- it turns out to be identical to the simplest case of our algorithm for skew-circulant matrix-vector product.
\end{example}

\section{Tensor rank, border rank, and bilinear complexity}\label{sec:rank}

The \textit{Strassen tensor rank method} that we have alluded to in the introduction studies the optimal bilinear complexity of a bilinear operation by studying the rank and border rank of its structure tensor.
\begin{definition}
Let $\mu_\beta$ be the structure tensor of a bilinear map $\beta:U\times V \to W$, we say that the \textit{tensor rank} or just \textit{rank} of $\mu_\beta$ is $r$ if $r$ is the smallest positive integer such that there exist $u^*_1,\dots, u^*_r\in U^*$, $v^*_1,\dots, v^*_r\in V^*$, and $w_1, \dots, w_r\in W$ such that 
\begin{equation}\label{eq:decomp}
\mu_\beta=\sum_{i=1}^r u^*_i\otimes v^*_i\otimes w_i.
\end{equation}
We denote this by $\rank(\mu_\beta)=r$.
We say that the \textit{border rank} of $\mu_\beta$ is $r$ if $r$ is the smallest positive integer such that there exists a sequence of tensors $\{\mu_n \}_{n=1}^\infty$ of rank $r$ such that 
\[
\lim_{n\to \infty} \mu_n= \mu_\beta.
\]
We denote this by $\brank\mu_\beta=r$. We define the rank and border rank of the zero tensor to be zero.
\end{definition}

Our interest in tensor rank is that it gives us exactly  \textit{the least number of multiplications} required to evaluate $\beta(u,v)$ for arbitrary inputs $u$ and $v$. This is established later in Proposition~\ref{prop:rank multiplication}. In which case border gives the least number of multiplication required to evaluate $\beta(u,v)$  up to arbitrarily high accuracy for arbitrary inputs $u$ and $v$. The study of complexity of bilinear operations in this manner is called \textit{bilinear complexity}, originally due to Strassen \cite{Strass1, Strass2} and has developed into its own subfield within complexity theory \cite[Chapter~14]{BCS}.

To illustrate this, we will start with a simple analysis to show that the usual way of computing matrix-vector product has optimal bilinear complexity, i.e., computing the product of an $m \times n$ matrix and \change{a vector of dimension $n$} requires $mn$ multiplications and one cannot do better.
\begin{proposition}\label{prop:rank lower bound}
Let $\beta:U\times V\to W$ be a bilinear map and suppose $\operatorname{span}(\mu_\beta(U\otimes V))=W$. Then 
\[
\rank(\mu_\beta)\ge \dim W.
\]
The role of $W$ may be replaced by $U$ or $V$.
\end{proposition}
\begin{proof}
If not, then 
\[
\mu_\beta=\sum_{i=1}^r u^*_i\otimes v^*_i\otimes w_i\in U^*\otimes V^*\otimes W
\]
for some integer $r< \dim W$ and vectors $u^*_i\in U^*, v^*_i\in V^*, w_i\in W$. Hence 
\[
\operatorname{span}(\mu_\beta(U\otimes V))\subset \operatorname{span}\{w_i : i=1,\dots, r\}.
\]
But this contradicts the assumption that $\mu_\beta(U\otimes V)=W$.
\end{proof} 

As an immediate application of Proposition~\ref{prop:rank lower bound}, we will show that for a general matrix, the usual way of doing matrix-vector product is already optimal, i.e., Strassen-type fast algorithms  for matrix-matrix product do not exist when one of the matrices is a vector (has only one column or row).
\begin{corollary}\label{cor:mat-vec}
Let $\beta:\mathbb{C}^{m\times n}\times \mathbb{C}^n\to \mathbb{C}^m$ be the bilinear map defined by the matrix-vector product. Then 
\[
\rank(\mu_\beta)=mn.
\]
\end{corollary}
\begin{proof}
It is easy to see that $\rank(\mu_\beta)\le mn$. On the other hand, \[
\mu_\beta(\mathbb{C}^n\otimes (\mathbb{C}^m)^*)=(\mathbb{C}^{m\times n})^*,
\]
since for the matrix $E_{ij}$ with $(i,j)$th entry one and zero elsewhere, we have
\[
\mu_\beta(e_j\otimes e^*_i)(E_{ij})=e^*_i(E_{ij}e_j)=1.
\]
Here $\{e_j: j=1,\dots,n\}$ is the standard basis of $\mathbb{C}^n$ and $\{e^*_i: i=1,\dots,m\}$ is its dual basis for $(\mathbb{C}^m)^*$.
\end{proof}
This establishes our earlier claim that $mn$ is the minimal number of required multiplications for forming a matrix-vector product. Next we show that we cannot do better than $mn$ even if we are only interested in computing matrix-vector product up to arbitrary accuracy.

Clearly we have
\[
\brank(\mu_\beta)\le \rank(\mu_\beta)
\]
in general but equality is attained in the following special case. \change{The next proposition turns out to be a very useful result for us --- we will rely on it repeatedly to find border ranks of various structure tensors in Sections~\ref{sec:sparse}--\ref{sec:simul}.}
\begin{proposition}\label{prop:border rank equals rank}
Let $\beta:U\times V\to W$ be a bilinear map and assume $\mu_\beta(U\otimes V)=W$ and $\rank(\mu_\beta)=\dim W\le \dim U\dim V$. Then 
\[
\brank(\mu_\beta)=\rank(\mu_\beta).
\]
\end{proposition}
\begin{proof}
Assume that $\brank(\mu_\beta)<\rank(\mu_\beta)$. Notice that we may regard $\beta$ as a linear map 
\[
\beta: U\otimes V \to W.
\]
Since $\mu_\beta(U\otimes V)=W$ and $\rank(\mu_\beta)=\dim W$, the rank of $\beta$ as a linear map (or a matrix) is $\dim W$. Let $r'=\brank(\mu_\beta)$. Then there is a sequence $\{\mu_n\}_{n=1}^\infty$ of tensors of rank $r'$  such that 
\[
\lim_{n\to \infty} \mu_n=\mu_\beta.
\]
We can similarly regard $\mu_n$ as a linear map $\mu_n:U\otimes V \to W$. Since $\rank(\mu_n)=r'$ we see that the rank of $\mu_n$ as a linear map is at most $r'$. By the choice of the sequence $\{\mu_n\}_{n=1}^\infty$, we see that
\[
\lim_{n\to \infty} \mu_n = \mu_\beta
\]
as linear maps (or matrices). Hence the border rank of $\mu_\beta$ as a linear map (or a matrix) is at most $r'$. However, for matrices, the notion of rank is the same as border rank. This contradicts the fact that $\rank \mu_\beta = \dim W > r'$.  
\end{proof}

We deduce that the usual way of performing matrix-vector product is also optimal even if we are only interested in approximating the result up to arbitrary accuracy. \change{The following result may also be deduced from the proof (but not the statement) of \cite[Lemma~6.1]{Schon}.}
\begin{corollary}\label{cor:mat-vec2}
The border rank of the structure tensor of $m\times n$ matrix-vector product is $mn$.
\end{corollary}

We now give the deferred proof establishing the role of tensor rank in bilinear complexity. This simple result is well-known, classical (see the discussions in \cite{BCS, Strass,Strass1,Strass2}), \change{and has a trivial proof. But given its central importance in our article, we include the statement and proof for easy reference.}
\begin{proposition}\label{prop:rank multiplication}
The rank of $\mu_\beta$ equals the least number of multiplications needed to compute the bilinear map $\beta$. 
\end{proposition}
\begin{proof}
Given $u\in U$ and $v\in V$, then by definition 
\[
\mu_\beta(u,v,\cdot)=\beta(u,v)\in W. 
\]
Since $\rank(\mu_\beta)=r$ we may write $\mu_\beta$ as
\[
\mu_\beta=\sum_{i=1}^r u^*_i \otimes v^*_i\otimes w_i
\]
for some $u^*_i\in U^*,v^*_i\in V^*$ and $w_i\in W$. Hence 
\[
\beta(u,v)=\sum_{i=1}^r u^*_i(u)v^*_i(v) w_i\in W.
\]
Notice that $u^*_i(u)$ and $v^*_i(v)$ are complex numbers and thus to compute $\beta$ we only need $r$ multiplications.
\end{proof}

\subsection{Remarks on arithmetic}\label{sec:arith} We highlight a common pitfall in the precise meaning of the word `multiplication'  used in the context of bilinear complexity. Here it refers strictly to the multiplications of indeterminates but excludes multiplications of a constant and an indeterminate or of two constants. For example, we need one multiplication to calculate $(x,y) \mapsto x\cdot y$ but zero multiplication to calculate $x \mapsto cx$ for any constant $c \in \mathbb{C}$. We will use the term 'scalar multiplication' to refer to the multiplication of two scalars or that of a scalar and an indeterminate over the given field. For instance, $2\cdot 3$ or $2\cdot x$ each requires one scalar multiplication  to compute.

So in the context of bilinear complexity, a discrete Fourier transform (\textsc{dft}) may be computed with just $O(1)$, in fact zero, multiplications. The usual $O(n \log n)$ complexity in \textsc{fft} counts scalar multiplications. As the case of \textsc{fft} illustrates, one reason for the exclusion of multiplication involving constants is that this part may often be performed with specialized subroutines or implemented in hardware. On the other hand, bilinear complexity counts only multiplications of variable quantities that could change from input to input.

\change{In particular, traditional studies of structured matrix-vector product, e.g., the superfast algorithms in \cite[Chapter~2]{Pan}, rely on the usual measure of computational complexity, counting all arithmetic operations (addition, multiplication, scalar addition, scalar multiplication, etc) as opposed to bilinear complexity. That is why the complexity estimates in \cite[Chapter~2]{Pan} differ substantially from those in Sections~\ref{sec:circulant}--\ref{sec:hank}.}

In the most widely studied case of matrix multiplication $\beta: \mathbb{C}^{n \times n} \times \mathbb{C}^{n \times n} \to \mathbb{C}^{n \times n}$, the rank of $\mu_\beta$ informs us  about the total arithmetic complexity of $\beta$, i.e., counting both additions and multiplications of indeterminates, as the following theorem from \cite{BCS} shows.
\begin{theorem}[Strassen]\label{thm:matrix rank complexity}
Let $R_n$ be the rank of $\mu_\beta$ and let $M_n$ the computational complexity of the matrix multiplication. Then 
\[
\inf\{\tau\in \mathbb{R}: M_n = O(n^\tau)\}=\inf\{\tau \in \mathbb{R} : R_n=O(n^\tau)\}.
\]
\end{theorem}
This says that asymptotically, the order of rank of $\mu_\beta$ equals the order of the computational complexity of matrix multiplication. Moreover, combined with Proposition~\ref{prop:rank multiplication}, the number of multiplications needed in matrix multiplication dominates the number of additions. This is a very special phenomenon. In general, the number of multiplications cannot dominate the number of additions.

%\red{There is a related but different notion of \emph{multiplicative complexity} defined in \cite[Definition~4.7]{BCS}. The main distinction is that for bilinear complexity, the vector spaces $U,V,W$ do not come with prespecified bases, whereas for multiplicative complexity, these vectors spaces are assumed to be equipped with specific bases. Let $w_1,\dots, w_p$ be a basis of $W$ where $p=\dim W$. The image $\beta(u, v)$ is a linear combination $\beta(u, v) = \sum_{i=1}^p \lambda_i w_i$. Multiplicative complexity asks for the minimum number of multiplications required to compute $\lambda_1,\dots,\lambda_p$  whereas bilinear complexity asks for the minimum number of multiplications required to compute $\lambda_1,\dots,\lambda_p$ over all possible choices of bases of $W$. Note that as usual, a $\lambda_i$ that is known \textit{a priori} to be zero  does not need to be computed.}

\subsection{\change{Remarks on numerical stability}}\label{sec:stability} Numerical stability has been discussed extensively for Gauss's complex multiplication algorithm in \cite{Higham2}, for Strassen-style matrix multiplication algorithms in \cite{DDHK}, and for general bilinear operations in \cite{Miller}. Our goal in this section is to highlight the connection between numerical stability of a bilinear operation $\beta$ and the nuclear norm \cite{FL} of its structure tensor $\mu_\beta$.

Since numerical stability is an analytic notion, we will need to assume that $U^*$, $V^*$, and $W$ are norm spaces. For notational simplicity, we will denote the norms on all three spaces by $\|\cdot\|$.  Let $\beta:U\times V\to W$ be a bilinear map and $\mu_\beta \in U^* \otimes V^* \otimes W$ be its structure tensor. We will rewrite the tensor decomposition \eqref{eq:decomp} in the form
\begin{equation}\label{eq:decompalgo}
\mu_\beta = \sum_{i=1}^{r} \lambda_i u_{i}^*\otimes v_{i}^*\otimes w_{i}
\end{equation}
where $\lVert u_i^* \rVert = \lVert v_i^* \rVert =\lVert w_i \rVert =1$, $i =1,\dots, r$. As we saw in Proposition~\ref{prop:rank multiplication}, any $r$-term decomposition of the form \eqref{eq:decompalgo}, irrespective of whether $r$ is minimum or not, gives an explicit algorithm for computing $\beta$: For any input $u \in U$, $v \in V$, $\beta(u,v) \in W$ is computed as
\[
\beta(u,v) = \sum_{i=1}^{r} \lambda_i u_{i}^*(u) v_{i}^*(v) w_{i}.
\]
Since the coefficient $\lambda_i$ captures the increase in magnitude at the $i$th step, we may regard the sum\footnote{This is essential; \eqref{eq:sumcoeff} cannot be replaced by $\bigl(\sum_{i=1}^{r}\lvert\lambda_{i} \rvert^p  \bigr)^{1/p}$ for $p > 1$ or $\max_{i=1,\dots,r} \lvert \lambda_i \rvert$. See \cite[Section~3]{FL}.} of (magnitude of) coefficients,
\begin{equation}\label{eq:sumcoeff}
 \sum\nolimits_{i=1}^{r}\lvert\lambda_{i} \rvert ,
\end{equation}
as a measure of the numerical stability of  the algorithm corresponding to  \eqref{eq:decompalgo}.

As we saw in Proposition~\ref{prop:rank multiplication}, when $r$ is minimum, the tensor rank
\[
\rank(\mu_\beta) =\min\left\{r :
\mu_\beta=\sum\nolimits_{i=1}^{r} \lambda_i u_{i}\otimes v_{i}\otimes w_{i}\right\}
\]
gives the least number of multiplications needed to compute $\beta$. Analogously, the nuclear norm
\begin{equation}\label{eq:nuclear}
\lVert \mu_\beta \rVert_{\ast}=\inf\left\{  \sum\nolimits_{i=1}^{r}\lvert\lambda_{i} \rvert : \mu_\beta =\sum\nolimits_{i=1}^{r}\lambda_{i}u_{i}\otimes v_{i}\otimes w_{i}, \; r \in \mathbb{N}\right\}
\end{equation}
quantifies the optimal numerical stability of computing $\beta$.

The infimum in \eqref{eq:nuclear} is always attained by an $r$-term decomposition although $r$ may not be $\rank(\mu_\beta)$. For example \cite[Proposition~6.1]{FL},
the structure tensor of complex multiplication in \eqref{eq:cplx} has nuclear norm
\[
\lVert \mu_\mathbb{C} \rVert_* = 4,
\]
and is attained by the decomposition \eqref{eq:cplx1} corresponding to the usual algorithm for complex multiplication but not the decomposition \eqref{eq:cplx2} corresponding to Gauss's algorithm --- since the sum of coefficients (upon normalizing the factors) in \eqref{eq:cplx2} is $2(1+ \sqrt{2})$.  In other words, Gauss's algorithm is less stable than the usual algorithm. Nevertheless, in this particular instance, there is a decomposition
\[
\mu_\mathbb{C}  = \frac{4}{3}\biggl(\biggl[\frac{\sqrt{3}}{2}e_1+\frac{1}{2}e_2\biggr]^{\otimes 3}+ \biggl[-\frac{\sqrt{3}}{2}e_1+\frac{1}{2}e_2\biggr]^{\otimes 3} + (-e_2)^{\otimes 3} \biggr)
\]
that attains both $\rank(\mu_\mathbb{C})$ and $\lVert \mu_\mathbb{C} \rVert_*$, i.e., the corresponding algorithm is simultaneously optimal in bilinear complexity and numerical stability.

Numerical stability is a moderately complicated notion \cite{Higham1} and cannot in general be adequately captured by a single number. The sum of  coefficients in \eqref{eq:sumcoeff} captures one aspect of numerical stability --- it is a measure akin to the \textit{growth factor} in Gaussian elimination with a specific pivoting scheme. The nuclear norm of the structured tensor may then be regarded as an analogue of the minimum growth factor over all possible pivoting strategies.

\section{Tensor ranks of structure tensors of  algebras}\label{sec:rankstruct}

Let $\mathcal{A}$ be an algebra of dimension $n$. Let $a_1,\dots, a_n$ be a basis of $\mathcal{A}$ and $a^*_1,\dots, a^*_n$ be its dual basis for $A^*$. Recall that the structure constants $c^{k}_{ij}$ determine the multiplication operation in $\mathcal{A}$, which we denote by $m_\mathcal{A}:\mathcal{A}\times \mathcal{A} \to \mathcal{A}$,
\[
m_\mathcal{A}(a_i,a_j)=\sum_{k=1}^n c^k_{ij} a_k, \qquad i,j =1,\dots,n.
\] 
The structure tensor $\mu_\mathcal{A}\in \mathcal{A}^*\otimes \mathcal{A}^* \otimes \mathcal{A}$ is then
\[
\mu_\mathcal{A}=\sum_{i,j,k=1}^n c^k_{ij} a^*_i \otimes a^*_j \otimes a_k.  
\] 
Note that $\mu_\mathcal{A}$ does not depend on the choice of basis and neither does the tensor and border ranks of $\mu_\mathcal{A}$.

When $\mathcal{A} =\mathbb{C}^{n \times n}$, $\mu_\mathcal{A}=\mu_{n,n,n}$ is the Strassen matrix multiplication tensor for product of square matrices. Inspired by the Cohn--Umans approach \cite{CU1} that we will discuss in the next section, we would like to study the rank of $\mu_\mathcal{A}$ for an arbitrary algebra, with a view towards embedding an operation whose bilinear complexity is difficult to analyze into an algebra where the task is easier. The first question that one needs to answer is the relation between ranks of the respective multiplication tensors. \change{The following proposition appears in \cite[Proposition~14.12]{BCS} but is stated without a proof. While we  do not need to use this proposition, we provide a proof that we think is instructive for our tensor rank calculations in Sections~\ref{sec:circulant}--\ref{sec:skew-symmetric}.}
\begin{proposition}
If an algebra $\mathcal{A}$ can be embedded into another algebra $\mathcal{B}$, i.e., $\mathcal{A}$ is isomorphic to a subalgebra of $\mathcal{B}$, then $\rank(\mu_\mathcal{A})\le \rank(\mu_\mathcal{B})$.
\end{proposition}
\begin{proof}
Let $j:\mathcal{A}\hookrightarrow \mathcal{B}$ be an embedding of $\mathcal{A}$ into $\mathcal{B}$ as algebras\footnote{Later on in the article we will consider embedding of vector spaces into algebras.}. Then it induces a surjection $j^*:\mathcal{B}^*\rightarrowtail \mathcal{A}^*$ and thus a surjection 
\[
j^*\otimes j^*\otimes \id_\mathcal{B}: \mathcal{B}^*\otimes \mathcal{B}^*\otimes \mathcal{B}\rightarrowtail \mathcal{A}^*\otimes \mathcal{A}^*\otimes \mathcal{B}.
\]
Let $\delta \coloneqq j^*\otimes j^*\otimes \id_\mathcal{B}$. 
We claim that $\delta(\mu_\mathcal{B})=\mu_\mathcal{A}$ and to show this, it suffices to show that 
\[
\delta(\mu_\mathcal{B})(a,a',\cdot)=\mu_\mathcal{A}(a,a',\cdot)=m_\mathcal{A}(a,a'),
\]
which is obvious from the definition of $\delta$. Let $r$ be the rank of $\mu_\mathcal{B}$ and suppose $\mu_\mathcal{B}$ has a tensor decomposition 
\[
\mu_\mathcal{B}=\sum_{i=1}^r b^*_{i,1}\otimes b^*_{i,2}\otimes b_{i,3}
\]
where $b^*_{i,1}, b^*_{i,2}\in \mathcal{B}^*$ and $b_{i,3}\in B$ for $i=1,\dots,r$. For notational simplicity we identify the image of $j$ in $\mathcal{B}$ with $\mathcal{A}$, regarding $\mathcal{A}$ as a subalgebra of $\mathcal{B}$. Let $\mathcal{A}^{\perp}$ be a subspace of $\mathcal{B}$ so that 
\[
\mathcal{B}=\mathcal{A}\oplus \mathcal{A}^\perp.
\]
Then we have  $b_{i,3}=a_{i,3}+a^{\perp}_{i,3}$ for  $i=1,\dots,r$ and thus
\[
\mu_\mathcal{B}=\left(\sum_{i=1}^r b^*_{i,1}\otimes b^*_{i,2}\otimes a_{i,3}\right)+\left(\sum_{i=1}^r b^*_{i,1}\otimes b^*_{i,2}\otimes a^\perp_{i,3}\right).
\]
Since $\delta(\mu_B)=\mu_\mathcal{A}$, we conclude that 
\[
\delta\left(\sum_{i=1}^r b^*_{i,1}\otimes b^*_{i,2}\otimes a^\perp_{i,3}\right)=0
\qquad
\text{and} 
\qquad
\delta\left(\sum_{i=1}^r b^*_{i,1}\otimes b^*_{i,2}\otimes a_{i,3}\right)=\mu_\mathcal{A}.
\]
So we obtain an expression of $\mu_\mathcal{A}$ as a sum of $r$ rank-one terms,
\[
\delta\left(\sum_{i=1}^r b^*_{i,1}\otimes b^*_{i,2}\otimes a_{i,3}\right)=\sum_{i=1}^r j^*(b^*_{i,1})\otimes j^*(b^*_{i,2})\otimes a_{i,3}\in \mathcal{A}^*\otimes \mathcal{A}^*\otimes \mathcal{A},
\]
and therefore $\rank(\mu_\mathcal{A})\le \rank(\mu_\mathcal{B})$.
\end{proof}
The map $j^*:\mathcal{B}^*\rightarrowtail \mathcal{A}^*$ above may be viewed as the restriction of linear forms on $\mathcal{B}$ to $\mathcal{A}$. 

\begin{corollary}
If the algebra $\mathbb{C}^{n \times n}$ can be embedded into an algebra $\mathcal{B}$, then the computational complexity of multiplying two matrices is bounded by the rank of structure tensor $\mu_\mathcal{B}$ of $\mathcal{B}$.
\end{corollary}

If we fix a basis for $\mathcal{A}$ and $\mathcal{B}$ then we can identify $\mathcal{A}$ with its dual $\mathcal{A}^*$ and $\mathcal{B}$ with its dual $\mathcal{B}^*$. Hence we may regard $\mu_\mathcal{A}\in \mathcal{A}^*\otimes \mathcal{A}^*\otimes \mathcal{A}$ as an element of $\mathcal{B}^*\otimes \mathcal{B}^*\otimes \mathcal{B}$. We would like to compare the rank  of $\mu_\mathcal{A}$ as an element of $\mathcal{A}^*\otimes \mathcal{A}^*\otimes \mathcal{A}$ and  as an element of $\mathcal{B}^*\otimes \mathcal{B}^*\otimes \mathcal{B}$. We denote these by $\rank_\mathcal{A}(\mu_\mathcal{A})$ and $\rank_\mathcal{B}(\mu_\mathcal{A})$ respectively. We will rely on the following proposition found in \cite{de Silva/Lim}.
\begin{proposition}
Let $U_1, \dots, U_n$ be vectors spaces and let $U'_1, \dots, U'_n$ be linear subspaces of $U_1, \dots, U_n$ respectively. Let $T \in U'_1\otimes \cdots \otimes U'_n$. Suppose $T$ has rank $r'$ as an element in $U'_1\otimes \cdots\otimes U'_n$ and has rank $r$ as an element in $U_1\otimes \cdots \otimes U_n$, then $r=r'$.
\end{proposition}
\begin{corollary}
If an algebra $\mathcal{A}$ can be embedded into another algebra $\mathcal{B}$, then
\[
\rank_\mathcal{A}(\mu_\mathcal{A})=\rank_\mathcal{B}(\mu_\mathcal{A}).
\]
\end{corollary}
Even though rank stays unchanged under embedding of vector spaces over the same ground field, the same tensor may well have different ranks over different fields. For example  \cite{Chen/Bryinski}, 
\[
T=e_0\otimes(e_0\otimes e_0-e_1\otimes e_1)+e_1\otimes (e_0\otimes e_1+e_1\otimes e_0),
\]
has rank three (over $\mathbb{R}$) when viewed as an element of $(\mathbb{R}^2)^{\otimes 3}$, but it has rank two (over $\mathbb{C}$) when viewed as an element of $(\mathbb{C}^2)^{\otimes 3}$.

We state here a result of Winograd \cite{Knuth,Winograd}, rephrased slightly differently in terms of structure tensors of algebras.
\begin{theorem}[Winograd]\label{thm:Winograd}
Let $p(x)$ be a monic \change{polynomial} of degree $n$ whose complete factorization over a given infinite field $k$ is 
\[
p(x)=p_1(x)^{e_1}\dots p_{q}(x) ^{e_q}.
\]
Then the rank of the structure tensor of the algebra $k[x]/(p(x))$ is $2n-q$.
\end{theorem}

\section{Generalized Cohn--Umans method and tensor rank}\label{sec:cohn}

A major advance in the study of bilinear complexity of matrix multiplication is the \textit{Cohn--Umans group theoretic method} proposed in \cite{CU1}. The gist of this idea is that one may compute multiplication in the matrix algebra by `embedding' it in a judiciously chosen group algebra \cite{CU1} \change{or, more recently, in an adjacency algebra of a coherent configuration \cite{CU2}.} The embedding in the Cohn--Umans method is however not an embedding of algebras as in Section~\ref{sec:rankstruct}, but an embedding of vector spaces. As such one needs a certain `triple product property' to hold to ensure that the entries of the matrix product may still be read off from the entries of the element in the group algebra. In this section, we will generalize the Cohn--Umans method and relate it to tensor rank.

We start by briefly summarizing the Cohn--Umans method. We assume working over $\mathbb{C}$ but the discussions in this and the next section hold for any field. Let $G$ be a finite group and let $\mathbb{C}[G]$ denote its group algebra over $\mathbb{C}$.
\begin{theorem}[Cohn--Umans]
Suppose that $G$ contains subsets $S,T,U$ of cardinality $m,n,p$ respectively such that for any $s,s'\in S$, $t,t'\in T$, $u,u'\in U$, $stu=s't'u'$ implies $s=s'$, $t=t'$, $u=u'$. Then for matrices $A=(a_{ij})\in \mathbb{C}^{m \times n},B=(b_{ij})\in \mathbb{C}^{n \times p}$ we can associate elements $\widehat{A},\widehat{B}\in \mathbb{C}[G]$ as follows:
\begin{equation}\label{emb}
\widehat{A}=\sum_{i,j=1}^{m,n} a_{ij} s_i t_j^{-1},\qquad
\widehat{B}=\sum_{i,j=1}^{n,p} b_{ij} t_i u_j^{-1}.
\end{equation}
The $(i,j)$th entry of $AB$ is the coefficient of $s_iu_j^{-1}$ in $\widehat{A}\cdot\widehat{B}$. 
\end{theorem}
The condition in the first sentence of the theorem is called the \textit{triple product property}. If such a condition is met, we say that $G$ realizes $\langle m,n,p\rangle$.

The first step towards generalizing the Cohn--Umans method is to view the triple product property in an alternative manner, namely, it is equivalent (see Example~\ref{eg:tpp}) to saying that the following diagram commutes:
\begin{equation}\label{eq:trip1}
\begin{tikzcd}
\mathbb{C}^{m \times n} \otimes \mathbb{C}^{n \times p} \arrow{r}{j} \arrow[swap]{d}{m} &  \mathbb{C}[G]\otimes \mathbb{C}[G] \arrow{d}{m_G} \\
\mathbb{C}^{m \times p}  &\arrow{l}{\proj} \mathbb{C}[G]
\end{tikzcd}
\end{equation}
Here $j$ is the embedding of vector spaces defined by \eqref{emb}, $\proj$ is the projection map reading off entries of $AB$ from the coefficients of $\widehat{A}\cdot\widehat{B}$, and $m$ and $m_G$ are respectively the multiplications of matrices and elements in the group algebra.

To be more precise about the projection map, the translation of matrix multiplication into the multiplication elements in $\mathbb{C}[G]$ via \eqref{emb} introduces some `junk terms', i.e.,  $\widehat{A}\cdot\widehat{B}$ contains many coefficients that are not needed for obtaining the entries of $AB$. The projection map is simply a map that picks up the relevant coefficients. As will see in Section~\ref{sec:simul}, there are occasions when that those `junk terms' could turn out to be useful.

Another important feature of the Cohn--Umans method is that by Wedderburn Theorem, $\mathbb{C}[G]$ can be identified with a direct sum of matrix algebras of smaller sizes determined by the irreducible representations of $G$, giving an efficient way to compute the product $\widehat{A}\cdot\widehat{B}$. Since Wedderburn Theorem holds not just for group algebras but for any semismiple algebras, one may in principle use any semisimple algebra $\mathcal{A}$  in place of $\mathbb{C}[G]$. However, motivated by later examples, we will not insist that $\mathcal{A}$ be semisimple --- there will be occasions when it is useful to allow $\mathcal{A}$ to have nilpotent elements.

The commutative diagram view of the triple product property \eqref{eq:trip1} allows us to generalize it on an abstract level to arbitrary bilinear operations. Let $\beta:U\times V \to W$ be a bilinear map and let $\mathcal{A}$ be an algebra.  If there is an injective linear map $j:U\otimes V \to \mathcal{A}\otimes \mathcal{A}$ and a linear map $\proj:\mathcal{A}\to W$ such that the following diagram commutes:
\begin{equation}\label{eq:gentpp0}
\begin{tikzcd}
U \otimes V \arrow{r}{j} \arrow[swap]{d}{\beta} &  \mathcal{A} \otimes \mathcal{A}  \arrow{d}{m_{\mathcal{A}}} \\
W   &\arrow{l}{\proj} \mathcal{A}
\end{tikzcd}
\end{equation}
then we can translate the computation of $\beta$ into multiplication in the algebra $\mathcal{A}$. If this is the case, we will say that the algebra $\mathcal{A}$ \textit{realizes} the bilinear map $\beta$.

In case the reader is wondering why $\beta : U \times V \to W$ became $\beta : U \otimes V \to W$ in \eqref{eq:gentpp0}, recall that we do not distinguish between a bilinear map and its corresponding linear map, as we had explained after Proposition~\ref{prop:bilinear-linear}. We will adopt this convention in the rest of the article without further elaboration.

For readers unfamiliar with such constructions, we used tensor product rather than product in diagrams like \eqref{eq:gentpp0} because we want to preserve linear structures. If we had used product in \eqref{eq:gentpp0}, then $\beta:U\times V\to W$ has to be a bilinear map but this does not make sense in the category of vector spaces over $\mathbb{C}$ as morphisms in this category are linear maps (and bilinear maps are in general not linear).

Note that the embedding of $U\otimes V$ into $\mathcal{A}\otimes \mathcal{A}$ and the projection from $\mathcal{A}$ onto $W$ incur zero computational costs in the context of bilinear complexity (no multiplication involved). Hence we obtain
\[
\rank(\mu_\beta) \le 
\rank(\mu_\mathcal{A}).
\]

There are occasions (see Sections~\ref{sec:symm} and \ref{sec:skew-symmetric}) when it is useful to allow a more general framework where the multiplication of the algebra $\mathcal{A}$ is replaced by another bilinear map in  \eqref{eq:gentpp0}.
\begin{equation}\label{eq:gentpp1}
\begin{tikzcd}
U \otimes V \arrow{r}{j} \arrow[swap]{d}{\beta} &  U' \otimes V'  \arrow{d}{\beta'} \\
W   &\arrow{l}{\proj} W'
\end{tikzcd}
\end{equation}
\change{For readers familiar with modules \cite{AW, Lang}, a further generalization of \eqref{eq:gentpp1} is to have free modules over a ring in place of vector spaces over a field, i.e.,}
\begin{equation}\label{eq:gentpp2}
\begin{tikzcd}
M \otimes_R N \arrow{r}{j} \arrow[swap]{d}{\beta} &  M' \otimes_R N'  \arrow{d}{\beta'} \\
L   &\arrow{l}{\proj} L'
\end{tikzcd}
\end{equation}
\change{where $M,N,M',N',L,L'$ are free modules over a ring $R$, $\otimes_R$ denotes tensor product of $R$-modules, and all maps are morphisms of $R$-modules with $j$ injective.}

The commutative diagrams  \eqref{eq:gentpp0}, \eqref{eq:gentpp1}, \eqref{eq:gentpp2} and  may be viewed as a \textit{generalized triple product property}. They are similar to \cite[Lemma~14.10]{BCS} but we require $j$ to be injective. Another difference is that we require an embedding of  $U\otimes V \to U'\otimes V'$ and this may not necessarily arise from an embedding of $U \times V \to U' \times V'$ as in  \cite[Lemma~14.10]{BCS}.

We will call \change{the commutative diagrams \eqref{eq:gentpp0}, \eqref{eq:gentpp1}, or \eqref{eq:gentpp2} \emph{generalized Cohn--Umans method}.}
We will see several concrete realizations of this abstract framework later.  \change{Most of our applications of the generalized Cohn--Umans method will involve \eqref{eq:gentpp0} but there is one occasion (in Section~\ref{sec:skew-symmetric}) where we need the more general version in \eqref{eq:gentpp1} and another (in Example~\ref{eg:intmult}) where we need \eqref{eq:gentpp2}.}  In following we will establish some existential guarantees for this framework.

\begin{theorem}\label{thm: realization of any bilinear map}
Let $k$ be a field. Then every bilinear map over $k$ can be realized by an algebra $\mathcal{A}$ over $k$ where the rank of the structure tensor of $\mathcal{A}$ is equal to the rank of the structure tensor of the bilinear map.
\end{theorem}
\begin{proof}
Let $\beta:U\times V \to W$ be a bilinear map and let $\mu_\beta$ be the structure tensor of $\beta$. Without loss of generality we may assume that $\beta$ is surjective since otherwise we may replace $W$ by the image of $\beta$. Assume that $\rank(\mu_\beta)=r$ and that $\mu_\beta$ has a decomposition
\[
\mu_\beta=\sum_{i=1}^r u^*_i\otimes v^*_i\otimes w_i
\]
for some $u^*_i\in U^*, v^*_i\in V^*, w_i\in W$. Let $\mathcal{A} = k^r$, equipped with entrywise addition and multiplication. Consider the map
\[
j:U\times V\to k^r\times k^r 
\]
defined by $j(u,v)=\bigl((u^*_1(u),\dots,u^*_r(u), v^*_1(v),\dots,v^*_r(v)\bigr)$. Lastly define the projection map
\[
\proj: k^r\to W, \quad \sum_{i=1}^r x_i e_i \mapsto \sum_{i=1}^r x_i  w_i,
\]
where $\{e_i : i=1,\dots, r\}$ is the standard basis of $k^r$. It is clear that we have a commutative diagram
\[
\begin{tikzcd}
U \otimes V \arrow{r}{j} \arrow[swap]{d}{\beta} &  k^r \otimes k^r  \arrow{d}{m_{k^r}} \\
W   &\arrow{l}{\proj} k^r
\end{tikzcd}
\]
and hence $\beta$ is realized by the algebra $\mathcal{A}=k^r$.
\end{proof}

\begin{corollary}\label{cor:realization of any bilinear map over alg. closed field}
If $k$ is an algebraically closed field then every bilinear map over $k$ can be realized by the algebra $k[x]/(x^r-1)$ where $r$ is the rank of the structure tensor of the bilinear map.
\end{corollary}
\begin{proof}
Apply Theorem~\ref{thm: realization of any bilinear map} and Wedderburn Theorem.
\end{proof}

One may argue that Theorem~\ref{thm: realization of any bilinear map} and Corollary~\ref{cor:realization of any bilinear map over alg. closed field} are essentially tautological and not particularly useful since the algebra $\mathcal{A}$ involved is commutative (bearing in mind that the Cohn--Umans method becomes uninteresting when the group $G$ is abelian). We will consider another construction that yields a noncommutative algebra.  \change{The following construction gives a step-by-step recipe that starts from any given bilinear map $\beta$ and produces $\mathcal{A}$, a \textit{polynomial identity ring} or \textit{PI ring} \cite{McConnell/Robson}, for the generalized Cohn--Umans, i.e., \eqref{eq:gentpp0} is automatically a commutative diagram for this choice of $\mathcal{A}$.}

\begin{construction}\label{construction}
Let $\{u_1,\dots,u_m\}$ and $\{v_1,\dots,v_n\}$ be bases of $U$ and $V$. Let $\{u^*_1,\dots,u^*_m\}$ and $\{v^*_1,\dots,v^*_n\}$ be the corresponding dual bases of $U^*$ and $V^*$. We may assume that $\beta$ is nondegenerate since otherwise we may simply replace $U$ or $V$ by an appropriate subspace. In which case we may write the structure tensor $\mu_\beta\in U^* \otimes V^* \otimes W$ as
\[
\mu_\beta=\sum_{(i,j)\; :\; \beta(u_i,v_j)\ne0} u^*_i \otimes v^*_j \otimes w_{ij},
\]
i.e., the sum runs over all pairs of $(i,j)$ such that $\beta(u_i,v_j)\ne0$.  Let $k\langle x_1,\dots, x_m, y_1,\dots, y_n\rangle$ be the free algebra over $k$ with generators $x_1,\dots, x_m$ and $y_1,\dots, y_n$. We consider an ideal $I_0$ generated by the relations 
\[
L(x_iy_j)\sim 0 \qquad\text{if and only if} \qquad L(w_{ij})=0,
\]
where $L$ is a linear form over $k$ in $mn$ variables. Next let $I\supset I_0$ be an ideal of $k\langle x_1,\dots, x_m, y_1,\dots, y_n\rangle$ such that $I$ does not contain linear forms in $x_1,\dots,x_m$ or in $y_1,\dots, y_n$ and whenever there is a linear form $L$ with $L(x_iy_j)\in I$, $L(x_iy_j)\in I_0$ where at least one $(i,j)$ is a pair of indices such that $w_{ij}\ne 0$ in the representation of $\mu_\beta$. Then it is easy to verify that the above commutative diagram holds for $\mathcal{A}=k\langle x_1,\dots, x_m, y_1,\dots, y_n\rangle/I$, by sending $u_i$ to $x_i$ and $v_j$ to $y_j$. Such an algebra $\mathcal{A}$ is called a polynomial identity  or PI ring.
\end{construction}

\change{We will see in Section~\ref{sec:toep} how Construction~\ref{construction} can be used to obtain the optimum algorithm for Toeplitz matrix-vector product. In principle, it could also be used it to obtain the optimum algorithms for Hankel, symmetric, and Hankel-plus-Toeplitz matrix-vector product independently. However the relations between these structures have allowed us to build upon the algorithms that we obtained earlier and avoid repetitive use of Construction~\ref{construction}.}

\begin{example}[Triple Product Property]\label{eg:tpp}
Here $U=\mathbb{C}^{m \times n}$, $V=\mathbb{C}^{n \times p}$, $W=\mathbb{C}^{m \times p}$, we may take $\{E_{ij}: i =1,\dots,m; j =1,\dots,n\}$, $\{E_{ik}: i =1,\dots,m; k =1,\dots,p\}$ to be the standard basis for $U$ and $V$. The structure tensor may be expressed as
\[
\beta_{m,n,p}=\sum_{i,k=1}^{m,p} \left(\sum_{j=1}^n E^*_{ij}\otimes E^*_{jk}\right) \otimes E_{ik}.
\]
We see that $\{E_{ik} : i =1,\dots,m; k =1,\dots,p\}$ is a basis of $W$ and so there is no linear relation among them except the trivial relation $E_{ik}=E_{ik}$. Hence we obtain an ideal $I_0$ generated by $x_{ij}y_{jk}-x_{ij'}y_{j'k}$ for all choices of $i,j,j',k$. As we claim above, fany ideal $I$ such that $I\supset I_0$ does not contain linear forms in $x_{ij}$ or $y_{ij}$, thus the commutative diagram holds.

Now let us understand the triple product property in this framework. Let $G$ be a group with subsets $S,T,U$ satisfying the triple product property. Then it is obvious that  in $k[G]$ there is no nontrivial linear relation for elements of the form $s_it^{-1}_j$ and $t_ju^{-1}_k$. The triple product property guarantees that $(s_it^{-1}_j)(t_ku^{-1}_l)=(s_{i}t^{-1}_j)(t_ku^{-1}_l)$ if and only if $i=i'$, $j=j'$, $k=k'$. Hence for such $G$ the commutative diagram holds.
\end{example}
Our construction of  $\mathcal{A}$ from the structure tensor is purely formal. Usually all we could say is that $\mathcal{A}$ is a polynomial identity ring, which does not tell us a lot. However, in special circumstances, when the ideal $I$ is suitably chosen, we can obtain algebras with well-understood properties \change{that we} can exploit, as we will see in the rest of this paper.

Before beginning our discussion of structured matrix computations, we give an example to show the broad applicability of the generalized Cohn--Umans method.
\begin{example}[\change{Fast Integer Multiplication}]\label{eg:intmult}
The algorithms of Karatsuba \cite{Kara}, Toom--Cook \cite{Toom}, \cite[pp.~51--77]{Cook}, Sch\"{o}nhage--Strassen \cite{SS}, F\"{u}rer \cite{Furer},
are all instances of \eqref{eq:gentpp2}.
In these algorithms, $\mathbb{Z}$ is embedded in the $p$-adic integers $\mathbb{Z}_p$, and integers are then multiplied via $p$-adic multiplication of their images. Consider the commutative diagram of $\mathbb{Z}$-modules
\begin{equation}\label{eq:integer multiplication 1}
\begin{tikzcd}
\mathbb{Z} \otimes_\mathbb{Z} \mathbb{Z} \arrow{r}{j_p} \arrow[swap]{d}{\beta} &  \mathbb{Z}[x] \otimes_\mathbb{Z} \mathbb{Z}[x]  \arrow{d}{\beta'} \\
\mathbb{Z}   &\arrow{l}{\operatorname{ev}_p} \mathbb{Z}[x]
\end{tikzcd}
\end{equation}
where $\beta$ and $\beta'$ are multiplications in $\mathbb{Z}$ and $\mathbb{Z}[x]$ respectively.  For any $n\in \mathbb{Z}$,  we define $f_n(x) \coloneqq \sum_{i=0}^{d}a_i x^i \in \mathbb{Z}[x]$ where $a_0,\dots, a_d \in \{0,\dots,p-1\}$ are such that $n = \sum_{i=0}^d a_i p^i$, the base-$p$ (i.e., $p$-adic) expansion of $n$. The embedding  $j_p$ is defined by 
\[
j_p (m \otimes n) = f_m(x) \otimes f_n(x),
\] 
and the evaluation map  $\operatorname{ev}_p$ sends $f(x)\in \mathbb{Z}[x]$ to $f(p)\in \mathbb{Z}$. Now we may use divide-and-conquer, interpolation, discrete Fourier transform, and fast Fourier transform to multiply the two polynomials, giving us Karatsuba, Toom--Cook, Sch\"{o}nhage--Strassen, and F\"{u}rer algorithms respectively.
\end{example}

\section{Sparse, banded, and triangular matrices}\label{sec:sparse}

We begin our study of structured matrices with \textit{sparse matrices}, a particularly simple case that does not require the use of Cohn--Umans method. The results are also unsurprising.

One might wonder what happens to Corollaries~\ref{cor:mat-vec} and \ref{cor:mat-vec2} when the matrix involved has zero entries. The answer is what one would expect --- the tensor and border rank are both given by the number of nonzero entries. For any $\Omega \subseteq \{1,\dots,m \} \times \{1, \dots, n\}$, the set of matrices with \textit{sparsity pattern} $\Omega$ is
\[
\mathbb{C}^{m \times n}_\Omega \coloneqq \{ A \in \mathbb{C}^{m \times n} : a_{ij} = 0 \; \text{for all} \; (i,j) \not\in \Omega \},
\]
which is clearly a vector space of dimension $\# \Omega$.
\begin{proposition}\label{prop:sparse}
Let $\beta_\Omega : \mathbb{C}^{m \times n}_\Omega \times \mathbb{C}^n \to \mathbb{C}^m$, $(A, x) \mapsto Ax$. Let $\mu_\Omega \in (\mathbb{C}^{m \times n}_\Omega)^* \otimes (\mathbb{C}^n)^* \otimes \mathbb{C}^m$ be the corresponding structure tensor. Then
\[
\brank(\mu_\Omega) = \rank(\mu_\Omega) = \# \Omega.
\]
\end{proposition}
\begin{proof}
Since the dimension of $\mathbb{C}^{m \times n}_\Omega $  is $\# \Omega$ and the usual matrix-vector product costs $\# \Omega$ multiplications, the required result follows from Propositions~\ref{prop:rank lower bound}, \ref{prop:border rank equals rank}, and \ref{prop:rank multiplication}.
\end{proof}

The set of $n \times n$ \textit{banded matrices}  with upper bandwidth $k$ and lower bandwidth $l$ is the special case when
\[
\Omega = \{ (i,j) \in   \{1,\dots,n \} \times \{1, \dots, n\} : k < j - i < l\}.
\]
Such matrices are called \textit{diagonal} if $(k,l) = 0$, \textit{lower bidiagonal} if $(k,l) =(0,1)$, \textit{upper bidiagonal} if $(k,l) = (1,0)$, \textit{tridiagonal} if $(k,l)=(1,1)$, \textit{pentadiagonal} if $(k,l)=(2,2)$, \textit{lower triangular} if $(k,l)=(0,n-1)$, \textit{upper triangular} if $(k,l) = (n-1,0)$.
\begin{corollary}
The rank and border rank of the structure tensor of $n\times n$ banded matrix-vector product are both $n + [(n-1) +(n-2) + \dots + (n-k)] +  [(n-1) +(n-2) + \dots + (n-l) ]$.
\end{corollary}
\begin{corollary}\label{cor:triangular}
The rank and border rank of the structure tensor of $n\times n$ upper (or lower) triangular matrix-vector product are both $n(n+1)/2$.
\end{corollary}

\section{Circulant matrix}\label{sec:circulant}

In this section, we consider the problem of computing the product of a circulant matrix with a vector or, equivalently (as we shall see), the product of two circulant matrices. We will obtain an algorithm with optimal bilinear complexity (i.e., minimum number of multiplications). This algorithm will turn out to \change{be similar to} the well-known algorithm  (e.g.\ \cite[Section~4.8.2]{GVL} \change{or \cite[Section~2.4]{Pan}}) for computing product of two circulant matrices using \textsc{fft} \cite{fft} but we will derive it (i) purely from the perspective of optimal bilinear complexity and (ii) using the Cohn--Umans group theoretic method. \change{However, a key difference is that while the well-known  algorithm in \cite{GVL, Pan} crucially depends on the use of \textsc{fft}, our algorithm is indifferent to how \textsc{dft} is computed --- as we have pointed out in Section~\ref{sec:arith}, \textsc{dft} incurs no cost in bilinear complexity.} This serves as our first example of using the Cohn--Umans method for problems other than matrix multiplication.

We begin by considering a special case of  matrix multiplication where we multiply a $1 \times n$ matrix and an $n \times 1$ matrix. Let $a^\mathsf{T} =[a_1,\dots,a_n]\in \mathbb{C}^{1 \times n}$ be a row vector and let $b=[b_1,\dots,b_n]^\mathsf{T} \in \mathbb{C}^{n \times 1}$ be a column vector. It is easy to verify that the cyclic group $C_n=\langle g\mid g^n=1 \rangle$ realizes $\left\langle1,n,1 \right\rangle$ via subsets $S= \{1\}$, $T= C_n$,  $U = \{1\}$ that clearly satisfy the triple product property. By \eqref{emb}, we have 
\[
\widehat{a}=\sum_{i=1}^n a_i g^i,\qquad
\widehat{b}=\sum_{i=1}^n b_i g^{-i}.
\]
The coefficients of $g^{k}$ where $k=0,1,\dots, n-1$ are
\begin{equation}\label{eq:eq1}
\sum_{i=1}^n a_{i+k}b_i, \qquad \text{where}\; a_{s}=a_{s'} \; \text{iff} \; s\equiv s' \mod  n.
\end{equation}
To calculate $a^\mathsf{T}b$ we just need the coefficient of $1\in \mathbb{C}[C_n]$ but not the coefficients of the remaining $n-1$  terms (what we called `junk terms' earlier). On the other hand, if we calculate the product of two circulant matrices, then these $n-1$ terms become useful. 

Let $\Circ_n(\mathbb{C})$ be the linear space of all circulant matrices. It is well-known that $\Circ_n(\mathbb{C})$ is closed under the matrix multiplication and so is an algebra.
\begin{proposition}\label{prop:circulant}
Let $\beta_c : \Circ_n(\mathbb{C})\times \mathbb{C}^n \to \mathbb{C}^n$, $(\Circ(x), y) \mapsto \Circ(x)y$ be the circulant matrix-vector product and $\mu_c \in \Circ_n(\mathbb{C})^*\otimes (\mathbb{C}^n)^* \otimes \mathbb{C}^n$ be the corresponding structure tensor.
Let $\mu_{C}\in \Circ_n(\mathbb{C})^*\otimes \Circ_n(\mathbb{C})^* \otimes \Circ_n(\mathbb{C})$ be the structure tensor of the algebra $\Circ_n(\mathbb{C})$. Then 
\[
\rank(\mu_c)=
\rank(\mu_C)=n.
\]
\end{proposition}
\begin{proof}
A circulant matrix is completely specified by its first column or first row.
In particular, since the product of two circulant matrices is still circulant, the product is determined by its first column. Let $x=[x_1,\dots, x_n]^\mathsf{T} \in \mathbb{C}^n$  and let $\Circ(x)$ denote the circulant matrix 
\[
\Circ(x)=
\begin{bmatrix}
x_1 & x_2 & \dots & x_{n-1} & x_n\\
x_n & x_1 & \dots & x_{n-2} &x_{n-1}\\
\vdots & \vdots & \ddots & \vdots& \vdots \\
x_3 & x_4 & \dots & x_1 & x_2\\
x_2 & x_3 & \dots &  x_n &x_1 
\end{bmatrix} \in  \mathbb{C}^{n \times n}.
\]
Observe that to calculate the matrix-matrix product $\Circ(x)\Circ(y)$ it suffices to calculate the matrix-vector product $\Circ(x)[y_1,y_n,\dots, y_2]^\mathsf{T}$. This implies that the structure tensor of the algebra $\mu_{C}$ can be obtained from the structure tensor $\mu_{c}$ of the bilinear map $\beta_c$ and that 
\[
\rank(\mu_{c})=\rank(\mu_{C}).
\]
To compute $\rank(\mu_{c})$, observe that for two vectors $x,y\in \mathbb{C}^n$, we have
\[
\Circ(x)y= \left[\sum_{i=1}^n x_iy_{i}, \sum_{i=1}^n x_iy_{i+1},\dots, \sum_{i=1}^n x_iy_{i+n-1}\right]^\mathsf{T}.
\]
Here we adopt the same convention in \eqref{eq:eq1} that $y_s=y_{s'}$ iff $s\equiv s' \mod n$. Since the entries of $\Circ(x)y$ are exactly the coefficients of the product $\widehat{x}\cdot\widehat{y}\in \mathbb{C}[C_n]$, where $\widehat{x},\widehat{y} \in \mathbb{C}[C_n]$  are obtained as in \eqref{emb}. Now it remains to count the number of multiplications needed to form $\widehat{x}\cdot\widehat{y}$ in $\mathbb{C}[C_n]$. Since $C_n$ is the cyclic group of order $n$, it has exactly $n$ representations, all of dimension  one; these are indexed by the roots of unity $1, \omega,\dots, \omega^{n-1}$ where $\omega=e^{2k\pi \mathrm{i}/n}$. Denote these representations by $V_0,\dots, V_{n-1}$ where $V_i\simeq \mathbb{C}$ is given by\footnote{We do not distinguish between an irreducible representation of $G$ and its irreducible $\mathbb{C}[G]$-submodule.}
\[
\rho_i: C_n\to \mathbb{C}, \qquad g\mapsto \omega^i, \qquad i =0,\dots,n-1.
\]
On the other hand, by Wedderburn Theorem we have 
\[
\mathbb{C}[C_n]\simeq \bigoplus_{i=0}^{n-1} V_i^*\otimes V_i,
\]
i.e., we may express elements in $\mathbb{C}[C_n]$ as $n\times n$ \change{diagonal} matrices. Explicitly, $\widehat{x}=\sum_{i=0}^{n-1} x_i g^i $ corresponds to the diagonal matrix
\[
\diag \left(\sum_{i=0}^{n-1} x_i\omega^i, \sum_{i=0}^{n-1} x_i\omega^{2i},\dots, \sum_{i=0}^{n-1} x_i\omega^{(n-1)i},\sum_{i=0}^{n-1} x_i\right),
\]
and $\widehat{y}=\sum_{i=0}^{n-1} y_i g ^{-i}$ corresponds to the diagonal matrix
\[
\diag \left(\sum_{i=0}^{n-1} y_i\omega^{-i}, \sum y_i\omega^{-2i},\dots, \sum_{i=0}^{n-1} y_i\omega^{-(n-1)i},\sum_{i=0}^{n-1} y_i\right).
\]
Therefore we need  $n$ multiplications to compute $\widehat{x}\cdot\widehat{y}$ and thus
\[
\rank(\mu_{c})\le n.
\]
On the other hand, it follows from Proposition~\ref{prop:rank lower bound} that $\rank(\mu_{C})\ge n$ since the image $\mu_{C}(\Circ_n(\mathbb{C})\otimes \Circ_n(\mathbb{C}))$ is the whole of $\Circ_n(\mathbb{C})$.
\end{proof}
The proof of Proposition~\ref{prop:circulant} is constructive --- it gives an algorithm with optimal bilinear complexity that computes a circulant matrix-vector product or a circulant matrix-circulant matrix product using only $n$ multiplications. In fact, this algorithm is essentially the same as the well-known algorithm for circulant matrix-vector product using Fast Fourier Transform (\textsc{fft}).

A departure from usual considerations in numerical linear algebra is that we only care about the number of  multiplications used in the algorithm. We minimize the number of multiplications by paying the price of using more additions. We require $n^2$ additions to execute our algorithm if we have our input $(\Circ(x), y)$ and output $\Circ(x)y$ expressed in the standard basis $\{e_1,\dots,e_n\}$ on $\mathbb{C}^n$. However, if we use the Fourier basis $f_1,\dots,f_n$ on $\mathbb{C}^n$, i.e., the Discrete Fourier Transform (\textsc{dft}) of $e_1,\dots,e_n$, then we require no addition at all to execute our algorithm.

The \textsc{dft} is a linear map and so computing $f_1,\dots,f_n$ from $e_1,\dots,e_n$ would involve only additions and scalar multiplications. Hence the use of different bases will not change the number of multiplications needed to multiply a circulant matrix to a vector (or two circulant matrices). This agrees with our expectation --- a tensor and therefore its rank do not depend on the choice of bases.

Proposition~\ref{prop:border rank equals rank} immediately gives us the border rank analogue of Proposition~\ref{prop:circulant}. It is also straightforward to obtain the analogue of Proposition~\ref{prop:circulant} for inversion of circulant matrices.
\begin{corollary}
The border ranks of the structure tensor of the bilinear operation $\beta_{c}$ and the structure tensor of the algebra $\Circ_n(\mathbb{C})$ are both $n$.
\[
\brank(\mu_{c})=\brank(\mu_{C})=n.
\]
\end{corollary}

\begin{corollary}\label{cor:circulant inverse}
Let $X=\Circ(x)$ be a nonsingular circulant matrix. Then one requires just $n$ divisions to compute its inverse $X^{-1}$. 
\end{corollary}
\begin{proof}
Let $X = \Circ(x)$ where $x = [x_1,\dots, x_n]^\mathsf{T} \in \mathbb{C}^n$ and let $Y = X^{-1}$ be given by $Y = \Circ(y)$ where $y = [y_1,\dots, y_n]^\mathsf{T} \in \mathbb{C}^n$. As in the proof of Proposition~\ref{prop:circulant}, their corresponding images in $\mathbb{C}[C_n]$ are
\begin{align*}
\widehat{X}&=\diag \left(\sum_{i=0}^{n-1} x_i\omega^i, \sum_{i=0}^{n-1} x_i\omega^{2i},\dots, \sum_{i=0}^{n-1} x_i\omega^{(n-1)i},\sum_{i=0}^{n-1} x_i\right),\\
\widehat{Y}&=\diag \left(\sum_{i=0}^{n-1} y_i\omega^{-i}, \sum y_i\omega^{-2i},\dots, \sum_{i=0}^{n-1} y_i\omega^{-(n-1)i},\sum_{i=0}^{n-1} y_i\right).
\end{align*}
Since $\widehat{X} \cdot \widehat{Y}=I$, we obtain
\[
\widehat{Y} = \diag \Biggl( \left(\sum_{i=0}^{n-1} x_i\omega^i\right)^{-1}, \left(\sum_{i=0}^{n-1} x_i\omega^{2i}\right)^{-1},\dots, \left(\sum_{i=0}^{n-1} x_i\omega^{(n-1)i}\right)^{-1},\left(\sum_{i=0}^{n-1} x_i\right)^{-1}\Biggr).
\]
Hence inverting $X$ requires $n$ divisions\footnote{The reader is reminded that scalar multiplications by a constant like $\omega^i$ are not counted in bilinear complexity.}.
\end{proof}

\section{$f$-circulant and skew-circulant matrices}\label{sec:f-circ}

We now extend the work of the previous section to \textit{$f$-circulant matrices}.  For any $f \in \mathbb{C}$, an $f$-circulant matrix is one of the form
\[
\begin{bmatrix}
x_1 & x_2 & \dots & x_{n-1} & x_n\\
fx_n & x_1 & \dots & x_{n-2} & x_{n-1}\\
\vdots & \vdots & \ddots & \vdots & \vdots\\
fx_3 & fx_4 & \dots & x_1 & x_2\\
fx_2 & fx_3 & \dots & fx_n & x_1
\end{bmatrix} \in  \mathbb{C}^{n \times n}.
\]
We denote the vector space of $n \times n$ $f$-circulant matrices by $\Circ_{n,f}(\mathbb{C})$. Evidently, a $1$-circulant matrix is just a usual circulant matrix. If $f=-1$, an $f$-circulant matrix is also called a \textit{skew-circulant matrix}.

It is \change{well-known \cite[Theorem~2.6.1]{Pan} and} straightforward to see that $\Circ_{n,f}(\mathbb{C})\simeq \mathbb{C}[x]/(x^n-f)$ and therefore also an algebra. We may employ the same techniques we used in the case $f=1$ to prove the following.
\begin{proposition}
The rank and border rank of the structure tensor of the $f$-circulant matrix-vector product over $\mathbb{C}$ are both $n$. Furthermore, one can invert a nonsingular $f$-circulant matrix over $\mathbb{C}$ using just $n$ divisions.
\end{proposition}
The proofs of these statements are near identical to those in the previous section and we will not repeat them. What we will instead investigate is an interesting special case when $n=2$ and $f = -1$ but over $\mathbb{R}$ instead of $\mathbb{C}$.
\begin{proposition}\label{prop:rank real skew-circulant}
The rank of the structure tensor of $2\times 2$ skew-circulant matrix-vector product over $\mathbb{R}$ is three.
\end{proposition}
\begin{proof}
The product of a $2 \times 2$ real skew-circulant matrix with a vector is given by
\[
X=\begin{bmatrix}
a & b\\
-b & a
\end{bmatrix} \in \mathbb{R}^{2 \times 2},\qquad
v=\begin{bmatrix}
c \\
-d
\end{bmatrix} \in \mathbb{R}^2,\qquad
Xv=\begin{bmatrix}
ac-bd\\
-ad-bc
\end{bmatrix} \in \mathbb{R}^2.
\]
Observe that to compute $Xv$ we require just three real multiplications:
\[
M_1=(a+b)(c+d),\qquad
M_2=ac,\qquad
M_3=bd
\]
to obtain
\begin{equation}\label{eq:gauss2}
ac-bd=M_2-M_3,\qquad
-ad-bc=-(M_1-M_2-M_3).
\end{equation}
Therefore, the rank $r$ of the structure tensor of $2\times 2$ skew-circulant matrix-vector product over $\mathbb{R}$ is at most three. We show that $r$ cannot be two.  Suppose $r=2$, then there exist polynomials $M_1,M_2$ in $a,b,c,d$, each costing only one multiplication to evaluate, such that 
\[
ac-bd=\alpha_1 M_1 + \alpha_2 M_2,\qquad
ad+bc=\beta_1 M_1 + \beta_2 M_2 .
\]
Since $a,b,c,d$ are independent variables, we must also have that
\[
\det \begin{bmatrix}
\alpha_1 & \alpha_2\\
\beta_1 & \beta_2 
\end{bmatrix} \ne 0.
\]
Hence $M_1$ and $M_2$ are both linear combination of $ac-bd$ and $ad+bc$. In particular, there exist $s,t \in \mathbb{R}$ such that 
\[
M_1=s (ac-bd)+t (ad+bc).
\]
But  as $M_1$  can only involve one multiplication, it must have the form
\[
M_1=(s_ 1 a +s_2b + s_3 c + s_4 d) (t_1a +t_2b + t_3c + t_4d).
\]
Without loss of generality we may assume that $s_1=0$ and $t_1=1$. Then $s_2=0$, $s_3=s$, and $s_4=t$. These imply that $t_2=t/s$ and $t_3=t_4=0$ and thus
\[
M_1=(s c + t d)\left(a+\frac{t}{s} b\right),
\] 
giving us $t^2/s=-s$, a contradiction since both $s$ and $t$ are real.
\end{proof}

We show a somewhat unexpected link between Proposition~\ref{prop:rank real skew-circulant} and Example~\ref{eg:cplx}.
\begin{corollary}
The rank and border rank of $\mu_{\mathbb{C}}$, the structure tensor of $\mathbb{C}$ as an $\mathbb{R}$-algebra, or equivalently, the structure tensor of the $\mathbb{R}$-bilinear map 
\[
\beta:\mathbb{C}\times \mathbb{C}\to \mathbb{C}, \quad (a+bi, c + di) \mapsto (ac -bd, ad +bc),
\]
are both three.
\end{corollary}
\begin{proof}
Let $z_1 = a + bi$ and $z_2 = c+di$. If we identify $\mathbb{C}$ with $\mathbb{R}^2$, then $z_1z_2=Xv$ where $X$ and $v$ are as defined in Proposition~\ref{prop:rank real skew-circulant}. From which it is clear that the structure tensor $\mu_{\mathbb{C}} \in (\mathbb{R}^2)^*\otimes (\mathbb{R}^2)^*\otimes \mathbb{R}^2$ has rank three over $\mathbb{R}$. The conclusion regarding border rank follows from \cite[Theorem 7.1]{de Silva/Lim}.
\end{proof}
One may check that the optimal algorithm for skew-circulant matrix-vector product in \eqref{eq:gauss2} is in fact the same as Gauss's method for multiplication of complex numbers \eqref{eq:gauss}.

\section{Toeplitz matrices}\label{sec:toep}

Let $\Toep_n(\mathbb{C})$ be the vector space of $n\times n$ Toeplitz matrices. The following  result is well-known, proved in \cite{Bini/Capovani} using methods different from those we employ below. Our objective of including this is to provide another illustration of the generalized Cohn--Umans approach where a bilinear operation is embedded in an algebra, in this case, the algebra of circulant matrices  $\Circ_{2n}(\mathbb{C})$ in Section~\ref{sec:circulant}.
\begin{theorem}[Bini--Capovani]\label{thm:toeplitz rank}
Let $\beta_t : \Toep_n(\mathbb{C}) \times \mathbb{C}^n \to \mathbb{C}^n$ be the Toeplitz matrix-vector product. Let $\mu_{t} \in \Toep_n(\mathbb{C})^*\otimes (\mathbb{C}^n)^* \otimes \mathbb{C}^n$ be the structure tensor of $\beta_t$. Then
\[
\rank(\mu_{t})=2n-1.
\]
\end{theorem}

\begin{proof}
We begin by observing that there is an embedding of an $n\times n$ Toeplitz matrix $X_n=[x_{j-i}]$ as a subblock of a $2n\times 2n$ circulant matrix $C_{2n}$ as follows
\begin{equation}\label{eq:toepcirc}
C_{2n}=\begin{bmatrix}
X_n & Y_n\\
Y_n & X_n
\end{bmatrix},
\end{equation}
where 
\[
Y_n=\begin{bmatrix}
y & x_{-n+1} & x_{-n+2} & \dots & x_{-1} \\
x_{n-1} & y & x_{-n+1} & \dots & x_{-2}\\
x_{n-2} & x_{n-1} & y & \dots & x_{-3}\\
\vdots & \vdots & \vdots & \ddots & \vdots\\
x_1 & x_2 & x_3 & \dots & y
\end{bmatrix}
\]
and $y \in \mathbb{C}$ can be arbitrarily chosen. We will choose
\[
y=-\sum_{i=-(n-1)}^{n-1} x_i.
\]
and by this choice of $y$, we only need $2n-1$ multiplications to compute the $2n\times 2n$ circulant matrix-vector product. To see this, recall that in our proof of  Proposition~\ref{prop:circulant}, the multiplication of $Y_n$ by a vector of dimension $2n$ is computed by multiplying a pair of $2n\times 2n $ diagonal matrices. In our case, the $2n\times 2n$ diagonal matrix corresponding to $Y_n$ is one whose $(1,1)$th entry is zero. Hence an $n\times n$ Toeplitz matrix-vector product can be computed with just $2n-1$ multiplications and so the rank of the corresponding structure tensor $\mu_{t}$ is at most $2n-1$. On the other hand, by Proposition~\ref{prop:rank lower bound}, $\rank(\mu_{t})$ is at least $2n-1$ since $\mu_{t}(\mathbb{C}^n\otimes (\mathbb{C}^n)^*)$ is the whole of $\Toep_n(\mathbb{C})^*$.
\end{proof}

As in Section~\ref{sec:circulant}, Proposition~\ref{prop:border rank equals rank} and Theorem~\ref{thm:toeplitz rank} together give the corresponding result for border rank.
\begin{corollary}\label{cor:toeplitz rank}
The structure tensor of the Toeplitz matrix-vector product has border rank
\[
\brank(\mu_t) = 2n-1.
\]
\end{corollary}

We embedded $\Toep_n(\mathbb{C})$ into $\Circ_{2n}(\mathbb{C})\simeq \mathbb{C}[C_{2n}]$ and the Toeplitz matrix-vector product inherits a group theoretic interpretation via this embedding. Given $X=[x_{j-i}] \in \Toep_n(\mathbb{C})$ and a vector $z=[z_1,\dots, z_n]^{\mathsf{T}} \in \mathbb{C}^n$, we may explicitly construct the product $Xz$ as follows. First construct two vectors
\[
a=[x_0,x_1,\dots, x_{n-1}, y , x_{-n+1}, x_{-n+2}, \dots, x_{-1} ]^\mathsf{T}  \in \mathbb{C}^{2n},\qquad
b=[z_1, \dots, z_n,0,\dots, 0]^{\mathsf{T}} \in \mathbb{C}^{2n},
\]
where $y=-\sum_{i=-(n-1)}^{n-1} x_i$. Notice that 
\[
a^\mathsf{T} b=\sum_{i=1}^{n} x_{i-1}z_i
\]
is the first entry of the vector $Xz$. As we had observed in Section~\ref{sec:circulant}, the cyclic group $C_{2n}=\left\langle g\mid g^{2n}=1 \right\rangle$ realizes $\left\langle 1, 2n,1 \right\rangle$. Hence we may construct two elements in $\mathbb{C}[C_{2n}]$, $\widehat{a}$ and $\widehat{b}$ as in Section~\ref{sec:circulant} by
\[
\widehat{a}=\sum_{i=1}^n x_{i-1} g^{i}+ y g^{n+1} + \sum_{i=n+2}^{2n} x_{i-2n-1}g^{i} \in \mathbb{C}[C_{2n}], \qquad
\widehat{b}=\sum_{i=1}^{n} z_i g^{-i} \in \mathbb{C}[C_{2n}].
\]
It is easy to see that coefficients of $g^{2n}, g^{2n-1},\dots, g^{n+1}$ in $\widehat{a}\cdot\widehat{b}$ give the required entries of $Xz$. 

We have seen in Corollary~\ref{cor:circulant inverse} that inverting a circulant matrix can be done with just $n$ divisions. Although we may embed $\Toep_n(\mathbb{C})$ into $\Circ_{2n}(\mathbb{C})$ via \eqref{eq:toepcirc} and $C_{2n}$ may be inverted with $2n$ division, there does not seem to be a way to obtain $X_n^{-1}$ from $C_{2n}^{-1}$.

\change{Suppose we are unaware of the fact that we may embed a Toeplitz matrix into a circulant matrix of larger size, how could we have discovered it?} We now provide an illustration of how Construction~\ref{construction} may be applied systematically to discover the appropriate algebra to use in the generalized Cohn--Umans method for Toeplitz matrix-vector product. Let $\{T_k \in \mathbb{C}^{n \times n}: k=1,\dots, 2n-1\}$ be the standard basis for the space of Toeplitz matrices  $\Toep_n(\mathbb{C})$, i.e., the entries of $T_k=[t_{ij}]$ are
\[
t_{ij}=
\begin{cases}
0 &\textit{if} \; j-i\ne k,\\
1 &\textit{if} \; j-i=k.
\end{cases}
\]  
Let $\{e_i \in \mathbb{C}^n :  i=1,\dots, n\}$ be the standard basis of $\mathbb{C}^n$. Then 
\[
T_k e_i=e_{n-k+i},
\]
where $e_{i}\coloneqq0$ whenever $i \ge n+1$ or $i \le 0$. As described in Section~\ref{sec:cohn}, we start from the free algebra  $\mathbb{C}\langle x_1,\dots, x_{2n-1},y_1,\dots,y_n \rangle$ and let $I_0$ be the ideal generated by the relations
\[
x_ky_i=x_{k'}y_{i'} \quad \textit{whenever}\quad 0\le k-i=k'-i'\le n-1,
\]
or, equivalently,
\begin{equation}\label{eqn:generator of I_0}
x_ky_1=x_{k+i-1}y_{i} \quad \textit{for all} \quad 1\le i \le n,\; 1\le k \le 2n-i.
\end{equation}

Next we construct an ideal $I$ such that (i) $I_0\subset I$, (ii) $I$ does not contain any linear form in $x_i$ or $y_j$, and (iii) whenever $F=L(x_ky_i)$ where $L$ is a linear form and at least one $x_ky_i$ appears in $F$ for some $0\le k-i\le n-1$, then $F\in I_0$. Without loss of generality, we may assume that $y_1=1$  and so \eqref{eqn:generator of I_0} simplifies to
\[
x_k=x_{k+i-1}y_i \quad \textit{for all} \quad 1\le i \le n, \; 1\le k \le 2n-i.
\]
A moment's thought would then lead us to taking $x_k=x^k$ where $x$ is such that $x^{2n}=1$, and also $y_i=x^{-i+1}=x^{2n-i+1}$ for $i=1,\dots, n$. It is straightforward to check the restrictions we imposed on $I$ are satisfied by these choices, which yield the algebra $\mathbb{C}[x]/(x^{2n}-1) \simeq \mathbb{C}[C_{2n}]$ that we seek.

We end this section with a brief word on Toeplitz matrix-Toeplitz matrix product $\beta_T : \Toep_n(\mathbb{C}) \times \Toep_n(\mathbb{C}) \to \mathbb{C}^{n \times n}$. Note that $\Toep_n(\mathbb{C})$ is not closed under matrix multiplication \cite{KL1}. The corollary below follows from Theorem~\ref{thm:toeplitz rank} and the fact that $XY = [Xy_1,\dots,Xy_n]$ for $X,Y \in \Toep_n(\mathbb{C})$ where $y_i$ is the $i$th column of $Y$.
\begin{corollary}\label{cor:toepliz mult}
The restriction  of the matrix multiplication tensor $\mu_{n,n,n}$ to the space $\Toep_n(\mathbb{C})$ of $n\times n$ Toeplitz matrices, regarded as a tensor in $\Toep_n(\mathbb{C})^*\otimes \Toep_n(\mathbb{C})^*\otimes \mathbb{C}^{n \times n}$, has rank  at most $n(2n-1)$.
\end{corollary}
From the perspective of iterative methods for Toeplitz matrices (both linear systems and least squares), understanding $\beta_t$ is more important than understanding $\beta_T$.

\section{Hankel matrices}\label{sec:hank}

The results in this short section follows from those in Section~\ref{sec:toep}. However we state them explicitly as these results on Hankel matrices are crucial for those on \textit{symmetric} matrices in Section~\ref{sec:symm}, which might come as a surprise. 

We introduce a few notations that we will use in Section~\ref{sec:symm}. Given a vector $x=[x_0,x_1,\dots,x_{2n-1}]^\mathsf{T} \in \mathbb{C}^{2n}$, we let 
\begin{equation}\label{eq:hank}
\Hank(x) \coloneqq 
\begin{bmatrix}
x_0 & x_1 & \dots & \dots & x_{n-2} & x_{n-1}\\
x_1 & \dots & \dots &\dots  & x_{n-1} & x_n\\
\vdots & \vdots & \vdots & \vdots & \vdots  &\vdots \\
x_{n-2} & x_{n-1} & \dots & \dots  & \dots & x_{2n-2}\\
x_{n-1} & x_n & \dots & \dots  &x_{2n-2} & x_{2n-1}
\end{bmatrix} \in  \mathbb{C}^{n \times n}
\end{equation}
be the Hankel matrix defined by $x$.
Let $\Hank_n(\mathbb{C})$ denote the vector space of $n\times n$ Hankel matrices. 

The corresponding results for Hankel matrices may be obtained from the ones for Toeplitz matrices essentially via the \change{well-known observation \cite[Theorem~2.1.5]{Pan}} that $X \in \mathbb{C}^{n \times n}$ is a Hankel matrix if and only if $JX$ and $XJ$ are both Toeplitz matrices. Here $J$ is the permutation matrix
\[
J\coloneqq\begin{bmatrix}
0 & 0 & \cdots & 0 & 1\\
0 & 0 & \cdots & 1 & 0\\
\vdots & \vdots & \ddots & \vdots & \vdots\\
0 & 1 & \cdots & 0 & 0\\
1 & 0 & \cdots &  0 & 0
\end{bmatrix} \in \mathbb{C}^{n \times n}.
\]
Since $J$ is a nonsingular linear transformation, $\mu_{h}$ and $\mu_{t}$ must have the same rank and border rank and we obtain the following from Theorem~\ref{thm:toeplitz rank}, Corollary~\ref{cor:toeplitz rank}, and Corollary~\ref{cor:toepliz mult}.
\begin{corollary}\label{cor:rank of Hankel matrix-vector product}
Let $\mu_{h} \in \Hank_n(\mathbb{C})^*\otimes (\mathbb{C}^n)^* \otimes \mathbb{C}^n$ be the structure tensor of the Hankel matrix-vector product $\beta_h : \Hank_n(\mathbb{C}) \times \mathbb{C}^n\to \mathbb{C}^n$.  Then
\[
\rank(\mu_{h}) =\brank(\mu_{h})=2n-1.
\]
Let $\mu_{H} \in \Hank_n(\mathbb{C})^*\otimes \Hank_n(\mathbb{C})^* \otimes \mathbb{C}^n$ be the structure tensor of the Hankel matrix-Hankel matrix product $\beta_H : \Hank_n(\mathbb{C}) \times \Hank_n(\mathbb{C})\to \mathbb{C}^{n \times n}$.  Then
\[
\rank(\mu_{H}) \le n(2n-1).
\] 
\end{corollary}
Since $\Hank_n(\mathbb{C})=J\Toep_n(\mathbb{C})=\Toep_n(\mathbb{C})J$, one expects a group theoretic realization of the Hankel matrix-vector multiplication. The construction is similar to that of the Toeplitz case.

\section{Triangular Toeplitz/Hankel matrices}\label{sec:triangular}

We include a discussion of triangular Toeplitz (or Hankel) matrix-vector product because the result may be somewhat unexpected --- its optimal bilinear complexity is exactly the same as that of a general Toeplitz (or Hankel) matrix-vector product. The fact that half the entries are zeros cannot be exploited to reduce the number of multiplications in an algorithm. Contrast this with Corollaries~\ref{cor:mat-vec} and \ref{cor:triangular}. Our methods in this section are new but the results are not, they follow from the work of Bini and Capovani \cite{Bini/Capovani}.

Let $\Toep_n^\Delta(\mathbb{C})$ be the linear space of $n\times n$ upper triangular Toeplitz matrices and let  $\beta_{\Delta} : \Toep_n^\Delta(\mathbb{C})\times \mathbb{C}^n \to \mathbb{C}^n$, $(A,v)\mapsto Av$ denote the upper triangular Toeplitz matrix-vector product.
We claim that the algebra $\mathcal{A}=\mathbb{C}[x]/(x^n)$ realizes $\beta_\Delta$. To see this, let
\[
A=\begin{bmatrix}
a_0 & a_1 & \dots & a_{n-1}\\
0    & a_0 & \dots & a_{n-2}\\
\vdots & \vdots & \ddots & \vdots\\
0 & 0 & \dots & a_0
\end{bmatrix} \in \Toep_n^\Delta(\mathbb{C}) \qquad \text{and}\qquad v=\begin{bmatrix}
v_0\\
v_1\\
\vdots \\
v_{n-1}
\end{bmatrix} \in  \mathbb{C}^n,
\]
we have 
\[
Av= \begin{bmatrix*}[r]
a_0 v_0 + a_1v_1 + \cdots + a_{n-1} v_{n-1}\\
a_0 v_1 + \cdots + a_{n-2} v_{n-1}\\
\ddots \qquad \qquad \\
a_0 v_{n-1}
\end{bmatrix*} \in  \mathbb{C}^n.
\]
Let $A_0,A_1,\dots,A_{n-1}$ be the `obvious' basis of $\Toep_n^\Delta(\mathbb{C})$, i.e., the $(i,j)$th entry of $A_k$ is one when $j-i=k$ and zero otherwise. Let $e_0,\dots, e_{n-1}$ be the standard basis of $\mathbb{C}^n$. We define an embedding of $\Toep_n^\Delta(\mathbb{C})\otimes \mathbb{C}^n$ into $\mathbb{C}[x]/(x^n) \otimes \mathbb{C}[x]/(x^n)$ taking the bases elements
\[
A_i\mapsto x^i, \qquad e_i\mapsto x^{n-1-i}, \qquad i =0,1,\dots,n-1.
\]
For $A=\sum_{i=0}^{n-1} a_i A_i\in \Toep_n^\Delta(\mathbb{C})$ and $v=\sum_{i=0}^{n-1} v_i e_i\in \mathbb{C}^n$, the images $\widehat{A},$ $\widehat{v}\in \mathbb{C}[x]/(x^n)$ are given by
\[
\widehat{A}=a_0 1 + a_1 x + \cdots + a_{n-1} x^{n-1},\qquad
\widehat{v}= v_0 x^{n-1} + v_1 x^{n-2}+ \cdots + v_{n-1}1.
\]
It is straightforward to verify that $\mathbb{C}[x]/(x^n)$ realizes $\beta_\Delta$. Note that $\mathbb{C}[x]/(x^n)$ is the cohomology ring of the complex projective space $\mathbb{CP}^{n-1}$. In particular it contains nilpotent elements and is not semisimple.

By Theorem~\ref{thm:Winograd}, the structure tensor of $\mathbb{C}[x]/(x^n)$ has rank $2n-1$, from which we may deduce the following.
\begin{theorem}\label{thm:upper triangular Toeplitz matrix-vector product}
Let $\mu_\Delta \in  \Toep_n^\Delta(\mathbb{C})^*\otimes (\mathbb{C}^n )^*\otimes \mathbb{C}^n$  be the structure tensor of the upper triangular Toeplitz matrix-vector product $\beta_\Delta$. Then
\[
\rank(\mu_\Delta) = 2n-1.
\]
\end{theorem}
Since $\Toep_n^\Delta(\mathbb{C})$ is a linear subspace of $\Toep_n(\mathbb{C})$, the structure tensor of upper triangular Toeplitz matrix-vector product is a projection of the structure tensor of Toeplitz matrix-vector product. However the tensor ranks of the two structure tensors are both $2n-1$.

\section{Toeplitz-plus-Hankel matrices}\label{sec:t+h}

Let $S_1$ and $S_2$ be two linear subspaces of $\mathbb{C}^{n\times n}$. Then the set $S_1 + S_2 = \{ X_1+ X_2 \in \mathbb{C}^{n \times n} : X_1 \in S_1, \; X_2 \in S_2\}$ is clearly also a linear subspace. If the structure tensors of the matrix-vector product for $S_1$ and $S_2$ have ranks $r_1$ and $r_2$ respectively, one might guess that the structure tensor of the matrix-vector product for $S_1+S_2$ has rank $r_1 + r_2$. However this is not true as we will see below.
\begin{example}
Let $\Toep_n^\Delta(\mathbb{C})$ be the linear subspace of upper-triangular Toeplitz matrices as in Section~\ref{sec:triangular}. Let  $\Toep_{n}^\Delta (\mathbb{C})^\mathsf{T}$ be the linear subspace of lower triangualr Toeplitz matrices. Clearly,
\[
\Toep_n^\Delta(\mathbb{C}) + \Toep_{n}^\Delta (\mathbb{C})^\mathsf{T}  = \Toep_n(\mathbb{C}).
\]
However, by Theorems~\ref{thm:toeplitz rank} and \ref{thm:upper triangular Toeplitz matrix-vector product}, the structure tensors of $\Toep_n^\Delta(\mathbb{C})$, $\Toep_{n}^\Delta (\mathbb{C})^\mathsf{T}$, and $\Toep_n(\mathbb{C})$ all have the same rank $2n-1$.  
\end{example}

In the special case $S_1 = \Toep_n(\mathbb{C}) $ and $S_2 = \Hank_n(\mathbb{C})$, a matrix in $S_1 + S_2$ is often called a \textit{Toeplitz-plus-Hankel} matrix \cite{Ng, Strang}. We show that the value of its rank is one less than the naive guess. 
\begin{proposition}
The structure tensor of the Toeplitz-plus-Hankel matrix-vector product has rank $4n-3$.
\end{proposition}
\begin{proof}
Let $E \in \mathbb{C}^{n \times n}$ be the matrix of all ones. For any $T\in \Toep_n(\mathbb{C})$ and $H\in \Hank_n(\mathbb{C})$ we have 
\[
T + H = (T+aE) + (H-aE)
\]
and  $T+aE \in\Toep_n(\mathbb{C})$, $H -aE \in \Hank_n(\mathbb{C})$ for all $a \in \mathbb{C}$.
We show that we may choose an appropriate $a\in\mathbb{C}$ so that the matrix-vector product for $T+aE$ requires only $2n-2$ multiplications. As in the proof of Theorem~\ref{thm:toeplitz rank}, we may embed $X=T+aE$ into a $2n\times 2n$ circulant matrix 
\[
C_{2n}=
\begin{bmatrix}
X & Y\\
Y & X
\end{bmatrix}
\]
that corresponds to a diagonal matrix whose $(1,1)$th entry is zero. We may choose $a\in\mathbb{C}$ so that the $(2,2)$th entry of this diagonal matrix is also zero. Hence the matrix-vector product with $T+aE$ costs at most $2n-2$ multiplications. Combined with Corollary~\ref{cor:rank of Hankel matrix-vector product}, we see that the structure tensor of  the matrix-vector product for $T+H$ has rank at most $4n-3$. On the other hand, we may check that $\Toep_n(\mathbb{C}) + \Hank_n(\mathbb{C})$ has dimension $4n-3$. So by Proposition~\ref{prop:rank lower bound}, the rank is exactly $4n-3$.
\end{proof}

\section{Block-Toeplitz-Toeplitz-block matrices}\label{sec:bttb}

One of the most common Toeplitz-like structure in numerical linear algebra is that of a \textit{block-Toeplitz-Toeplitz-block} or \textsc{bttb} matrix \cite{Kailath, Chan, Ng}. As the name suggests, these are $nk \times nk$ matrices that are $n\times n$ block Toeplitz matrices whose blocks are themselves $k \times k$ Toeplitz matrices, i.e.,
\[
A=
\begin{bmatrix}
X_0 & X_1 & \cdots & X_{n-2} &  X_{n-1}\\
X_{-1} & X_0 & \cdots & X_{n-3} & X_{n-2}\\
\vdots & \vdots & \ddots  & \vdots & \vdots \\
X_{2-n} & X_{3-n} & \cdots  & X_0 & X_1 \\
X_{1-n} & X_{2-n} & \cdots & X_{-1} &X_0
\end{bmatrix} \in \mathbb{C}^{nk \times nk},
\]
where $X_i \in \Toep_k (\mathbb{C})$ for all $i =-(n-1),\dots,-1,0,1,\dots,n-1$. We write $\bttb_{n,k}(\mathbb{C})$ for the set of $n \times n$ block Toeplitz matrices with $k \times k$ Toeplitz blocks.

There is of course nothing particularly special about the Toeplitz structure --- may also define block-Hankel-Hankel-block or \textsc{bhhb} matrices, block-circulant-circulant-block or \textsc{bccb} matrices, etc. In fact we will establish a general result that holds not only for any block matrices with structured blocks but those with \textit{multiple level block structures} (e.g., block Hankel matrices whose blocks are \textsc{bttb} matrices or block \textsc{bhhb} matrices whose blocks are \textsc{bccb} matrices).

For each $j=1,\dots,s$, let $U_{k_j}$ be a linear subspace of $\mathbb{C}^{k_j\times k_j}$. We define the following linear subspace
\[
 U_{k_1}\circledast \dots \circledast U_{k_s} \subseteq \mathbb{C}^{k_1 \times k_1} \circledast \dots \circledast \mathbb{C}^{k_s \times k_s}
\]
where `$\circledast$' denotes the \textit{Kronecker product} \cite{VL}. Note that
\[
\mathbb{C}^{k_1 \times k_1} \circledast \dots \circledast \mathbb{C}^{k_s \times k_s} = \mathbb{C}^{k_1\cdots k_s \times k_1 \cdots k_s}.
\]

In particular, the linear subspace of \textsc{bttb} matrices is obtained by setting $s = 2$, $k_1 = n$, $k_2 = k$, and $U_{k_1}  =  \Toep_n (\mathbb{C})$, $U_{k_2} =  \Toep_k (\mathbb{C})$, i.e.,
\[
 \Toep_n (\mathbb{C})  \circledast  \Toep_k (\mathbb{C}) = \bttb_{n,k}(\mathbb{C}).
\]
For $s=3$ and $U_{k_i}  =  \Toep_{k_i} (\mathbb{C})$, $i=1,2,3$, we obtain $k_1 \times k_1$ block Toeplitz matrices whose blocks are  $k_2k_3\times k_2k_3$ \textsc{bttb} matrices,
\[
\Toep_{k_1} (\mathbb{C})  \circledast  \Toep_{k_2} (\mathbb{C})  \circledast \Toep_{k_3} (\mathbb{C}) =  \Toep_{k_1} (\mathbb{C})  \circledast  \bttb_{k_2,k_3}(\mathbb{C}).
\]

\begin{lemma}\label{lem:block matrix}
Let $U \subseteq \mathbb{C}^{n \times n}$  and $V \subseteq \mathbb{C}^{k \times k}$ be linear subspaces. Let
\[
\beta_U : U \times \mathbb{C}^n \to \mathbb{C}^n,\qquad \beta_V : V \times \mathbb{C}^k \to \mathbb{C}^k, \qquad \beta_{U  \circledast  V} : (U  \circledast  V) \times \mathbb{C}^{nk} \to \mathbb{C}^{nk}
\]
be the corresponding matrix-vector products with respective structure tensors
\[
\mu_U \in U^* \otimes (\mathbb{C}^n )^* \otimes \mathbb{C}^n, \qquad \mu_V \in V^* \otimes (\mathbb{C}^k)^* \otimes \mathbb{C}^k, \qquad \mu_{U  \circledast  V} \in (U \circledast V)^* \otimes (\mathbb{C}^{nk} )^* \otimes \mathbb{C}^{nk}.
\]
Suppose
\begin{equation}\label{eq:condrank}
\rank(\mu_U)=\dim U, \qquad \rank(\mu_V) =\dim V
\end{equation}
and
\begin{equation}\label{eq:condsurj}
\mu_U(U \otimes \mathbb{C}^n) =\mathbb{C}^n, \qquad \mu_V(V \otimes \mathbb{C}^k) = \mathbb{C}^k.
\end{equation}
Then
\[
\rank(\mu_{U  \circledast  V}) = \rank(\mu_U)\rank(\mu_V).
\]
\end{lemma}
\begin{proof}
It is clear that $\rank(\mu_{U  \circledast  V})$ is bounded above by $ \rank(\mu_U)\rank(\mu_V)$. So it suffices to show that $\rank(\mu_{U  \circledast  V})$ is bounded below by $ \rank(\mu_U)\rank(\mu_V)$ but this  follows from applying Proposition~\ref{prop:rank lower bound} to the matrix-vector product
\[
(U  \circledast  V) \otimes \mathbb{C}^{nk} \to \mathbb{C}^{nk}. \qedhere
\]
\end{proof}

The desired result for \textsc{bttb} matrices follows immediately from Theorem~\ref{thm:toeplitz rank} and Lemma~\ref{lem:block matrix}.
\begin{corollary}
The rank of the structure tensor of the matrix-vector product  $\beta_{\bttb} : \bttb_{n,k}(\mathbb{C})  \times \mathbb{C}^{nk} \to \mathbb{C}^{nk} $ is $(2k-1)(2n-1)$.
\end{corollary}

We state a more general version of Lemma~\ref{lem:block matrix} that applies to multilevel block structures.
\begin{theorem}
For $j=1,\dots,s$, let $U_{k_j} \subseteq \mathbb{C}^{k_j \times k_j}$ be a linear subspace of Toeplitz, Hankel, $f$-circulant,  Toeplitz-plus-Hankel, symmetric, or sparse matrices (each $U_{k_j}$ may have a different structure). Let $\mu_j$ be the structure tensor of the matrix-vector product $U_{k_j} \times \mathbb{C}^{k_j} \to \mathbb{C}^{k_j} $, $j=1,\dots,s$, and let $\mu$ be that of  $(U_1  \circledast  \cdots  \circledast U_s) \times \mathbb{C}^{k_1 \cdots k_s} \to \mathbb{C}^{k_1 \cdots k_s}$. Then
\[
\rank(\mu) = \prod_{j=1}^s \rank(\mu_j).
\]
\end{theorem}
\begin{proof}
By our discussions in the previous and later sections, the conditions \eqref{eq:condrank} and \eqref{eq:condsurj} are met for these structured matrices. The result follows by applying Lemma~\ref{lem:block matrix} inductively.
\end{proof}
$f$-circulant matrices include circulant and skew-circulant ones;  sparse matrices include banded and triangular ones. Note that we have excluded skew-symmetric matrices and triangular Toeplitz matrices since they do not satisfy \eqref{eq:condsurj}.

Next we will discuss a Cohn--Umans realization of the matrix-vector product for $U_1  \circledast  \cdots  \circledast U_s$.:
\begin{proposition}
If the algebra $\mathcal{A}_j$ realize of the bilinear map $\beta: U_{k_j} \times \mathbb{C}^{k_j} \to \mathbb{C}^{k_j}$ for $j=1,\dots,s$, then the tensor product $\mathcal{A}= \mathcal{A}_1 \otimes \cdots \otimes \mathcal{A}_s$ realizes the Kronecker product $U_1  \circledast  \cdots  \circledast U_s$.
\end{proposition}
\begin{proof}
It suffices to prove the statement for $\beta_{U  \circledast  V} : (U  \circledast  V) \times \mathbb{C}^{nk} \to \mathbb{C}^{nk}$ when $\beta_U : U \times \mathbb{C}^n \to \mathbb{C}^n$ and  $\beta_V : V \times \mathbb{C}^k \to \mathbb{C}^k$ are realized by $\mathcal{A}$ and $\mathcal{B}$ respectively. But this follows from routine arguments: The embeddings $U \hookrightarrow \mathcal{A}$ and $V \hookrightarrow \mathcal{B}$ induce an embedding of $U  \circledast  V \hookrightarrow \mathcal{A} \otimes \mathcal{B}$ and the projections of $\mathcal{A}$ onto $\mathbb{C}^n$ and $\mathcal{B}$ onto $\mathbb{C}^k$ induce a projection of $\mathcal{A} \otimes \mathcal{B}$ onto $\mathbb{C}^{kn}$. The more general statement then \change{follows} from induction.
\end{proof}
For example the matrix-vector product $\beta_{\bttb} : \bttb_{n,k}(\mathbb{C})  \times \mathbb{C}^{nk} \to \mathbb{C}^{nk} $ is realized by the algebra
\[
\mathcal{A} = \mathbb{C}[x,y]/(x^{2k}-1,\; y^{2n}-1),
\]
from which we may also deduce the rank of the structure tensor of $\beta_{\bttb}$.

\section{Symmetric matrices}\label{sec:symm}

We saw in Corollaries~\ref{cor:mat-vec} and \ref{cor:mat-vec2} that the usual way of performing matrix-vector product is already optimal for general matrices. A natural question is: What if we require the matrix to be symmetric?  This is a very common situation since many, if not most, linear systems and least squares problems that arise in practice involve symmetric coefficient matrices. Despite this, we are unaware of any previous study. We show here that the optimal bilinear complexity for symmetric matrix-vector product is \change{$n(n+1)/2$}. Surprisingly the solution involves Hankel matrices.

We begin with the observation that every symmetric matrix may be expressed as a sum of symmetric Hankel matrices bordered by zeros. A $2 \times 2$ symmetric matrix is already a Hankel matrix. The $3 \times 3$ and $4 \times 4$ cases are shown explicitly below. 
\[
\begin{bmatrix}
a & b\\
b & c
\end{bmatrix},\qquad
\begin{bmatrix}
a & b & c\\
b & d & e\\
c & e & f
\end{bmatrix}
=\begin{bmatrix}
a & b & c\\
b & c & e\\
c & e & f
\end{bmatrix}
+
\begin{bmatrix}
0 & 0 & 0\\
0 & d -c & 0\\
0 & 0 & 0
\end{bmatrix},
\]
\[
\begin{bmatrix}
a & b & c & d\\
b & e & f & g\\
c & f & h & i \\
d & g & i & j
\end{bmatrix}
=
\begin{bmatrix}
a & b & c & d\\
b & c & d & g\\
c & d & g & i \\
d & g & i & j
\end{bmatrix}
+
\begin{bmatrix}
0 & 0 & 0 & 0\\
0 & e - c & f -d & 0\\
0 & f  -d & e -c & 0 \\
0 & 0 & 0 & 0
\end{bmatrix}
+
\begin{bmatrix}
0 & 0 & 0 & 0\\
0 & 0 & 0 & 0\\
0 & 0 & h -g -e + c & 0 \\
0 & 0 & 0 & 0
\end{bmatrix}.
\]
The generalization of this observation to $n \times n$ symmetric matrices will be established in our proof below, and together with Corollary~\ref{cor:rank of Hankel matrix-vector product}, be used to deduce the optimal bilinear complexity of symmetric matrix-vector product. Let $\mathsf{S}^2(\mathbb{C}^n)$ be the space of all $n\times n$ symmetric matrices. Let
$\beta_s:\mathsf{S}^2(\mathbb{C}^n)\times \mathbb{C}^n\to \mathbb{C}^n$
be the bilinear map of symmetric matrix-vector product and $\mu_s \in \mathsf{S}^2(\mathbb{C}^n)^* \otimes (\mathbb{C}^n)^* \otimes \mathbb{C}^n$.
\begin{theorem}\label{thm:rank of symmetric matrix-vector product}
The optimal bilinear complexity of symmetric matrix-vector product is \change{$n(n+1)/2$}, i.e., $\rank(\mu_s) = n(n+1)/2$.
\end{theorem}
\begin{proof}
By Proposition~\ref{prop:rank lower bound}, we see that $\rank(\mu_\beta)\ge \dim\mathsf{S}^2(\mathbb{C}^n)= n(n+1)/2$. On the other hand, for a given
\[
A=\begin{bmatrix}
a_{1,1} & a_{1,2} & \dots & a_{1,n-1} & a_{1,n}\\
a_{1,2} & a_{2,2} & \dots & a_{2,n-1} & a_{2,n}\\
\vdots & \vdots & \ddots & \vdots & \vdots\\
a_{1,n-1} & a_{2,n-1} & \dots & a_{n-1,n-1} & a_{n-1,n}\\
a_{1,n} & a_{2,n} & \dots & a_{n-1,n} & a_{n,n}  
\end{bmatrix} \in \mathsf{S}^2(\mathbb{C}^n)
\]
and a column vector $v=[v_1,\dots,v_n]^{\mathsf{T}} \in \mathbb{C}^n $, we claim that $(A,v) \mapsto Av$ can be computed as the sum of several Hankel matrix-vector products of decreasing sizes. 
Let
\[
H_1=\Hank(a_{1,1},\dots, a_{1,n},a_{2,n},\dots, a_{n,n}) \in \Hank_n(\mathbb{C}),
\]
notation as in \eqref{eq:hank}. Then $A-H_1$ is  a symmetric matrix of the form
\[
\begin{bmatrix}
0 & 0 & 0 \\
0 & A_2 & 0\\
0 & 0 & 0
\end{bmatrix} \in \mathsf{S}^2(\mathbb{C}^n),
\]
where $A_2 \in \mathsf{S}^{(n-2)}(\mathbb{C})$. Also we notice that
\[
(A-H_1)v = \begin{bmatrix}
0 \\ A_2v^{(2)}\\ 0
\end{bmatrix}
\]
where $v^{(2)}=[v_2,\dots, v_{n-1}]^{\mathsf{T}} \in \mathbb{C}^{n -2}$. Now we can repeat the above procedure and inductively we can prove our claim. Explicitly,
\[
Av=H_1v + \begin{bmatrix}
0 \\ H_2 v^{(2)} \\ 0
\end{bmatrix} + \begin{bmatrix}
0 \\ 0\\ H_3 v^{(3)} \\0 \\0
\end{bmatrix}+\cdots,
\]
where $H_i \in \Hank_{n-2i}(\mathbb{C})$ and $v^{(i)} \in \mathbb{C}^{n-2i}$, $i =1,\dots,\lfloor n/2 \rfloor$. $H_i$ and $v^{(i)}$ are linear in the entries of $A$ and $v$ respectively. By Corollary~\ref{cor:rank of Hankel matrix-vector product}, one can compute $H_i v^{(i)}$ in $2(n-2i)-1$ multiplications. Hence we obtain
\[
\rank(\mu_s)\le \sum_{i=0}^{\lfloor n/2 \rfloor} [2(n-2i)-1]=\frac{n(n+1)}{2}
\]
and therefore 
\[
\rank(\mu_s)=\frac{n(n+1)}{2}. \qedhere
\]
\end{proof}
One may also interpret the proof of Theorem~\ref{thm:rank of symmetric matrix-vector product} as an instance of the generalized Cohn--Umans method. We have an embedding of vector spaces
\begin{equation}\label{eq:symmhank}
\mathsf{S}^2(\mathbb{C}^n)^*\otimes (\mathbb{C}^n)^*\otimes \mathbb{C}^n\hookrightarrow \bigoplus_{i=0}^{\lfloor n/2 \rfloor} \Hank_{n-2i}(\mathbb{C})^*\otimes (\mathbb{C}^{n-2i})^*\otimes \mathbb{C}^{n-2i},
\end{equation}
and for each $i = 0, \dots, \lfloor n/2 \rfloor$, the bilinear map $\Hank_{n-2i}(\mathbb{C}) \times \mathbb{C}^{n-2i} \to \mathbb{C}^{n-2i}$, $(H_i, v^{(i)}) \mapsto H_i v^{(i)}$ is in turn realized by the algebra $\mathbb{C}[x]/(x^{2(n-2i)}-1)$. Note that the object on the right-hand side of \eqref{eq:symmhank} is not an algebra but only a vector space --- this is an application of the commutative diagram \eqref{eq:gentpp1}.

\section{Skew-symmetric matrices}\label{sec:skew-symmetric}

A departure from other sections in this article is that in this section we do not have the optimal bilinear complexity, only upper bounds. We first discuss the case of $3\times 3$ skew-symmetric matrix-vector product. Let
\[
A = \begin{bmatrix}
0 & a & b\\
-a & 0 & c\\
-b & -c & 0
\end{bmatrix} \in \mathsf{\Lambda}^2 (\mathbb{C}^3).
\]
Then the usual matrix-vector multiplication gives
\[
A\begin{bmatrix}
x\\
y\\
z 
\end{bmatrix}=\begin{bmatrix}
ay+bz\\
-ax+cz\\
-bx-cy
\end{bmatrix},
\] 
which costs six multiplications. So the rank of the structure tensor of the skew-symmetric matrix-vector product is at most six. We will rely on the following theorem \cite{Ottaviani,Strass} for the lower bound of the border rank (hence the rank) of a special $3$-tensor.
\begin{theorem}\label{thm:Ottaviani}
Let $T\in U\otimes V\otimes W$ where $\dim U =\dim V = \dim W =3$. Let $u_1,u_2,u_3$ be a basis of $U$. If we can write $T$ as
\[
T=u_1\otimes X_1 + u_2\otimes X_2 + u_3 \otimes X_3,
\]
with $X_1,X_2,X_3\in V\otimes W$ regarded\footnote{The result is however coordinate independent, i.e., it does not depend on our choice of the bases.} as $3\times 3$ matrices and if the following block matrix is nonsingular, 
\[
M_T =
\begin{bmatrix}
0 & X_3 & -X_2\\
-X_3 & 0 & X_1\\
X_2 & X_1 & 0
\end{bmatrix} \in \mathbb{C}^{9 \times 9},
\]
then $\brank(T)\ge 5$. The same result holds with $V$ or $W$ in the role of $U$.
\end{theorem}
Let $\mathsf{\Lambda}^2(\mathbb{C}^n)$ be the space of all $n\times n$ skew-symmetric matrices.
Note that $\dim \mathsf{\Lambda}^2 (\mathbb{C}^3)= 3$ and we may apply Theorem~\ref{thm:Ottaviani} to $T  = \mu_\mathsf{\Lambda}$, the structure tensor of the bilinear map
\[
\beta_\wedge : \mathsf{\Lambda}^2 (\mathbb{C}^3)\times \mathbb{C}^3\to \mathbb{C}^3, \quad \left(  \begin{bmatrix}
0 & a & b\\
-a & 0 & c\\
-b & -c & 0
\end{bmatrix},\begin{bmatrix} x\\ y\\ z \end{bmatrix}\right)\mapsto \begin{bmatrix} ay+bz\\-ax+cz\\ -bx-cy\end{bmatrix}.
\]
Let $e_1, e_2, e_3$ be the standard basis of $\mathbb{C}^3$ and
\[
F_1 = \begin{bmatrix}
0 & 1 & 0\\
-1 & 0 & 0\\
0 & 0 & 0
\end{bmatrix}, \quad
F_2 = \begin{bmatrix}
0 & 0 & 1\\
0 & 0 & 0\\
-1 & 0 & 0
\end{bmatrix}, \quad
F_3 = \begin{bmatrix}
0 & 0 & 0\\
0 & 0 & 1\\
0 & -1 & 0
\end{bmatrix}.
\]
be a basis of $ \mathsf{\Lambda}^2 (\mathbb{C}^3)$. Then we may decompose $\mu_\mathsf{\Lambda}$ as 
\[
\mu_\mathsf{\Lambda} =F_1\otimes (e_2\otimes e_1-e_1\otimes e_2)+ F_2\otimes (e_3\otimes e_1- e_1\otimes e_3) + F_3 \otimes (e_3\otimes e_2-e_2\otimes e_3)
\]
and it is easy to verify that $M_{\mu_\mathsf{\Lambda}} $ is nonsingular, giving us the following.
\begin{proposition}\label{prop:skew3}
The rank and border rank of the structure tensor of skew-matrix-vector product for $3 \times 3$ matrices are given by
\[
\rank(\mu_\mathsf{\Lambda})=5 \quad \text{or} \quad 6,
\]
and
\[
\brank(\mu_\mathsf{\Lambda})=5\quad \text{or} \quad 6.
\]
\end{proposition}
Next we construct an algebra that realizes the $3\times 3$ skew-symmetric matrix-vector product. Our candidate is
\[
\mathcal{A}=\mathbb{C}\left\langle x_1,x_2\right\rangle/(x_1^2,\; x_2^2, \; x_1x_2+x_2x_1).
\]
The embedding $\mathsf{\Lambda}^2 (\mathbb{C}^3)\times \mathbb{C}^3 \hookrightarrow \mathcal{A} \times \mathcal{A}$ is given by
\[
a_1\mapsto -x_1, \quad
a_2\mapsto -x_2, \quad
a_3\mapsto 1, \quad
b_1\mapsto 1, \quad
b_2\mapsto -x_2, \quad
b_3\mapsto x_1.
\]
Then given $A=\sum_{i=1}^3 u_i a_i\in \mathsf{\Lambda}^2 (\mathbb{C}^3)$ and $x=\sum_{i=1}^3 v_i b_i\in \mathbb{C}^3$, their images $\widehat{A},\widehat{x}\in\mathcal{A}$  are given by
\[
\widehat{A}=-u_1 x_1 - u_2 x_2 +u_3, \qquad
\widehat{x}= v_1 - v_2 x_2 + v_3 x_1,
\]
and their product is given by
\[
\widehat{A}  \cdot \widehat{x}=(-u_1 v_1 + u_3 v_3)\cdot x_1 + (u_1v_2+u_2v_3)\cdot x_1x_2 + (-u_2v_1-u_3v_2)\cdot x_2 + (u_3v_1) 1\in \mathcal{A}.
\]
Hence $\mathcal{A}$ realizes $3\times 3$ skew-symmetric matrix-vector product. We observe that $\mathcal{A}$ may be regarded as the cohomology ring of a torus, i.e., an exterior algebra of a two-dimensional vector space.

We now discuss the general case of $n\times n$ skew-symmetric matrix-vector product $\beta_\wedge : \mathsf{\Lambda}^2 (\mathbb{C}^n)\times \mathbb{C}^n\to \mathbb{C}^n$. We construct an algebra that realizes $\beta_\wedge$ starting with the inclusion of vector spaces
\begin{equation}\label{eq:wedge}
\mathsf{\Lambda}^2(\mathbb{C}^n)\hookrightarrow \bigl(\mathbb{C}[x]/(x^n+1)\bigr)\oplus W,
\end{equation}
where $W$ is a linear subspace of $\mathbb{C}^{n\times n}$ matrices satisfying the following conditions
\begin{enumerate}[\upshape (i)]
\item entries in the first row are all zeros;
\item diagonal entries are all zeros;
\item entries in the first column satisfy the relation $a_{i,1}+a_{n+2-i,1}=0$ for $i=2,\dots,n$.
\end{enumerate} 
Given $A \in 
\mathsf{\Lambda}^2(\mathbb{C}^n)$, the embedding is given by the decomposition
\[
A = \begin{bmatrix}
0 & a_{1,2} & \cdots & a_{1,n-1} & a_{1,n} \\
-a_{1,2} & 0  & \cdots &  a_{2,n-1} & a_{2,n}\\
\vdots & \vdots  & \ddots &\vdots & \vdots\\
-a_{1,n-1} & -a_{2,n-1}  & \cdots & 0 & a_{n-1,n}\\
-a_{1,n}  & -a_{2,n}  & \cdots & -a_{n-1,n} & 0
\end{bmatrix}= A_c +A_w, 
\]
where
\[
A_c=\begin{bmatrix}
0 & a_{1,2} & \cdots & a_{1,n-1} & a_{1,n} \\
-a_{1,n} & 0  & \cdots &  a_{1,n-2} & a_{1,n-1}\\
\vdots & \vdots  & \ddots &\vdots & \vdots\\
-a_{1,3} & -a_{1,4}  & \cdots & 0 & a_{1,2}\\
-a_{1,2}  & -a_{1,3}  & \cdots & -a_{1,n} & 0
\end{bmatrix} \in \Circ_{n, -1}(\mathbb{C})
\]
is a skew-circulant matrix and
\[
A_w =\begin{bmatrix}
0 & 0 & \cdots & 0 & 0 \\
-a_{1,2}+a_{1,n} & 0  & \cdots &  a_{2,n-1}-a_{1,n-2} & a_{2,n}-a_{1,n-1}\\
\vdots & \vdots  & \ddots &\vdots & \vdots\\
a_{1,3}-a_{1,n-1} & a_{1,4}-a_{2,n-1}  & \cdots & 0 & a_{n-1,n}-a_{1,2}\\
a_{1,2}-a_{1,n}  & a_{1,3}-a_{2,n}  & \cdots & a_{1,n}-a_{n-1,n} & 0
\end{bmatrix}\in W.
\]
Note that $A_c$, being skew-circulant, may be regarded as an element of $\mathbb{C}[x]/(x^n+1)$ as we have discussed in Section~\ref{sec:f-circ} and  we obtain the embedding in \eqref{eq:wedge}.

Since we also have
\[
\mathbb{C}^n \simeq \mathbb{C}[x]/(x^n+1),
\]
the bilinear map $\beta_\wedge : \mathsf{\Lambda}^2 (\mathbb{C}^n)\times \mathbb{C}^n\to \mathbb{C}^n$ may be realized as follows.
\[
\begin{tikzcd}
\mathsf{\Lambda}^2(\mathbb{C}^n) \otimes \mathbb{C}^n \arrow{r}{j} \arrow[swap]{d}{\id} &  \bigl(\mathbb{C}[x]/(x^n+1) \oplus  W\bigr) \otimes \mathbb{C}^n \arrow{d}{\simeq}\\
\mathsf{\Lambda}^2(\mathbb{C}^n) \otimes \mathbb{C}^n \arrow{r}{j} \arrow[swap]{d}{\id} &  \bigl(\mathbb{C}[x]/(x^n+1) \otimes \mathbb{C}^n \bigr)  \oplus (W \otimes \mathbb{C}^n)   \arrow{d}{\simeq}\\
\mathsf{\Lambda}^2(\mathbb{C}^n) \otimes \mathbb{C}^n \arrow{r}{j} \arrow[swap]{d}{\beta_\wedge} &  \bigl(\mathbb{C}[x]/(x^n+1) \otimes \mathbb{C}[x]/(x^n+1)\bigr) \oplus (W \otimes \mathbb{C}^n)  \arrow{d}{m}\\
\mathbb{C}^n   &\arrow{l}{\proj} \mathbb{C}[x]/(x^n+1)\simeq \mathbb{C}^n
\end{tikzcd}
\]
We have identified the skew-circulant matrix vector product with the skew-circulant matrix-matrix product (see Section~\ref{sec:f-circ}), i.e., the multiplication 
\[
\mathbb{C}[x]/(x^n+1) \times \mathbb{C}^n\to \mathbb{C}^n
\]
is identified with the multiplication
\[
\mathbb{C}[x]/(x^n+1) \times \mathbb{C}[x]/(x^n+1) \to \mathbb{C}[x]/(x^n+1).
\]
This realization is another instance of the commutative diagram \eqref{eq:gentpp1}. To put all these in concrete terms, given $A \in \mathsf{\Lambda}^2(\mathbb{C}^n) $ and $x \in \mathbb{C}^n$, we compute the matrix-vector product via
\[
Ax = \bigl(\text{first row of } A_c \Circ(x)\bigr) + A_w x.
\]
\begin{theorem}
The rank of the structure tensor of skew-symmetric matrix-vector product is bounded above by $n^2-n  - \lceil (n-1)/2 \rceil +1$. 
\end{theorem}
\begin{proof}
The first factor of the realization is the multiplication in the algebra
\[
\mathbb{C}[x]/(x^n+1) \times \mathbb{C}[x]/(x^n+1)\to \mathbb{C}[x]/(x^n+1).
\]
By  Theorem~\ref{thm:Winograd} we see that the rank of the structure tensor of the algebra $\mathbb{C}[x]/(x^n+1)$ is $n$. The second factor of the realization is a bilinear map
\[
W\times \mathbb{C}^{n} \to \mathbb{C}^{n}.
\]
A matrix-vector product  with a matrix in $W$ costs $n^2-(2n-1)-\lceil (n-1)/2 \rceil$ multiplications --- there are $2n-1$ zeros by (i) and (ii) and we invoke Proposition~\ref{prop:sparse};  moreover there are $ \lceil (n-1)/2 \rceil$ identical terms by (iii). Therefore this realization gives an upper bound  of
\[
n^2-n  - \lceil (n-1)/2 \rceil + 1. \qedhere
\]
\end{proof}
This upper bound is $n+\lceil (n-1)/2 \rceil - 1$ multiplications fewer than the usual matrix-vector product. In particular, for $n=3$ we obtain the upper bound in Proposition~\ref{prop:skew3}.

\section{Commutator}\label{sec:comm}

Our study of the bilinear complexity of commutators in this section  covers only the case of $2\times 2$ matrices. We do not yet know how to extend them to $n \times n$ matrices when $n > 2$.

We consider the bilinear map $[\cdot, \cdot]:\mathbb{C}^{2 \times 2}\times \mathbb{C}^{2 \times 2} \to \mathbb{C}^{2 \times 2}$ defined by $[A,X]=AX-XA$. We will write
\[
A=\begin{bmatrix}
a & b \\
c & d
\end{bmatrix}, \qquad
X=\begin{bmatrix}
x & y \\
z & w
\end{bmatrix},
\]
and therefore
\[
[A,X]=\begin{bmatrix}
bz-cy & (a-d)y-b(x-w)\\
-(a-d)z+c(x-w) & -(bz-cy)
\end{bmatrix}.
\]
Hence the rank of $\mu_{[\cdot,\cdot]} \in (\mathbb{C}^{2 \times 2})^*\otimes (\mathbb{C}^{2 \times 2})^* \otimes \mathbb{C}^{2 \times 2}$, the structure tensor of $[\cdot, \cdot ]$, is at most six. 

Now consider the matrix-vector product between matrices and vectors of the following forms
\[
\begin{bmatrix}
0 & -c & b\\
-b & a-d & 0\\
c & 0 & -a-d
\end{bmatrix} \in \mathbb{C}^{3 \times 3}, \qquad
\begin{bmatrix}
x-w\\
y\\
z
\end{bmatrix} \in \mathbb{C}^3.
\]
Notice that 
\[
\begin{bmatrix}
0 & -c & b\\
-b & a-d & 0\\
c & 0 & -(a-d)
\end{bmatrix}
\begin{bmatrix}
x-w\\
y\\
z
\end{bmatrix}=
\begin{bmatrix}
bz-cy\\
(a-d)y-b(x-w)\\
-(a-d)z+c(x-w)
\end{bmatrix},
\]
So the rank and border rank of the structure tensor $\mu_{[\cdot,\cdot]}$ of $[\cdot, \cdot]$ is the same as the rank of the following bilinear operation
\begin{equation}\label{eq:beta}
\beta: \mathbb{C}^3\times \mathbb{C}^3 \to \mathbb{C}^3, \quad
\left(
\begin{bmatrix}
s_1\\
s_2\\
s_3
\end{bmatrix},
\begin{bmatrix}
t_1\\
t_2\\
t_3
\end{bmatrix}
\right)\mapsto
\begin{bmatrix*}[r]
s_1t_2+s_2t_3\\
-s_2t_1+s_3t_2\\
-s_1t_1-s_3t_3
\end{bmatrix*},
\end{equation}
where
\[
s_1  = -c, \quad  s_2  = b, \quad s_3  = a-d, \quad
t_1  = x-w, \quad t_2 = y, \quad t_3  = z.
\]
We will need to distinguish the three copies of $\mathbb{C}^3$ in \eqref{eq:beta}, so for clarity let us denote them by $U$, $V$, and $W$ respectively, i.e.,
\[
\beta : U \times V \to W.
\]
Let $\{u_1, u_2, u_3 \}$, $\{v_1,v_2, v_3\}$, $\{w_1,w_2,w_3\}$ be the standard bases of $U,V,W$. Then  the structure tensor $\mu_\beta$ of $\beta$ may be decomposed as
\[
\mu_\beta=(u_1\otimes v_2+u_2\otimes v_3)\otimes w_1+ (-u_2\otimes v_1+u_3\otimes v_2)\otimes w_2 + (-u_1\otimes v_1-u_3\otimes v_3)\otimes w_3
\]
and we may apply Theorem~\ref{thm:Ottaviani} to obtain the following.
\begin{corollary}\label{cor: brank/rank of commutator}
The rank and border rank of the commutator for $2 \times 2$ matrices are given by
\[
\rank\bigl(\mu_{[\cdot,\cdot]}\bigr)=5 \quad \text{or} \quad 6
\]
and
\[
\brank\bigl(\mu_{[\cdot,\cdot]}\bigr)=5
\]
respectively.
\end{corollary}
In other words, for $A,X \in \mathbb{C}^{2 \times 2}$, computing $AX$ requires at least seven multiplications (e.g., Strassen's algorithm) whereas computing $[A,X] = AX - XA$ requires at most six multiplications. We suspect that this is always the case, i.e., computing commutator is always faster than computing matrix multiplication for $n \times n$ matrices.

We now construct an algebra $\mathcal{A}$ that realizes $\beta$ and therefore $[\cdot,\cdot]$. Let
\[
\mathcal{A} = \mathbb{C}\langle x_1,x_2 \rangle/(x_1^2,\;x_2^2,\; x_1x_2+x_2x_1)
\] 
and consider the embedding $U\otimes V \to \mathcal{A} \otimes \mathcal{A}$ induced by
\[
u_1 \mapsto x_1, \quad u_2 \mapsto x_2, \quad u_3 \mapsto 1,\quad
v_1 \mapsto -1, \quad v_2 \mapsto  x_2, \quad v_3 \mapsto -x_1.
\]
Given $s=\sum_{i=1}^{3} s_i u_i \in U$ and $t=\sum_{i=1}^3 t_i v_i \in V$, the images  $\widehat{s}$ and $\widehat{t}$ in $\mathcal{A}$ are 
\[
\widehat{s}=-s_1x_1 - s_2 x_2 + s_3 1,\qquad
\widehat{t}= -t_1 1 + t_2 x_2 - t_3 x_1
\]
respectively. Their product is 
\begin{align*}
\widehat{s} \cdot \widehat{t}&=(s_1x_1 + s_2 x_2 + s_3 1)(-t_1 1 + t_2 x_2 - t_3 x_1)\\
&=(-s_1t_1 -s_3t_3) x_1 + (s_1t_2+s_2t_3)x_1x_2 + (-s_2t_1+s_3t_2)x_2 + (-s_3t_1) 1,
\end{align*}
i.e., $\mathcal{A}$ realizes the bilinear map $[\cdot,\cdot]$. Observe that  $\mathcal{A}$ is the same algebra that we used to realize the $3\times 3$ skew-symmetric matrix-vector product in Section~\ref{sec:skew-symmetric}.

\section{Simultaneous matrix multiplication}\label{sec:simul}

We round out our list of bilinear operations with two examples of \textit{simultaneous matrix product}.
\begin{proposition}\label{prop:simultaneous matrix multiplication}
The following two matrix-matrix products:
\begin{equation}\label{eq:sim1}
\begin{bmatrix}
a & b\\
c & d
\end{bmatrix} \begin{bmatrix}
e & f\\
g & h
\end{bmatrix}
\qquad \text{and} \qquad \begin{bmatrix}
a & b\\
c & d
\end{bmatrix} \begin{bmatrix}
g & h\\
e & f
\end{bmatrix}
\end{equation}
can be computed simultaneously with eight multiplications.
\end{proposition}
\begin{proof}
Let $D_4=\langle x,y\mid x^4=y^2=1,\; yxy=x^{-1}\rangle$ be the dihedral group of order eight. The multiplication of $2\times 2$ matrices is realized by the subsets
\[
H_1  =\langle y \rangle =\lbrace y,1\rbrace, \qquad
H_2  =\langle x^2y \rangle=\lbrace x^2y,1 \rbrace, \qquad
S_3 =\lbrace x^{-1}y, 1\rbrace.
\]
Let
\begin{equation}\label{eq:ab}
A=\begin{bmatrix}
a & b \\
c& d
\end{bmatrix}, \qquad B=\begin{bmatrix}
e & f\\
g & h
\end{bmatrix}.
\end{equation}
Then $A,B$ correspond to $\widehat{A},\widehat{B} \in \mathbb{C}\left[ D_4\right]$ where 
\begin{align*}
\widehat{A}&=a\cdot (y^{-1}x^2y)+ b \cdot (y^{-1})+c \cdot (x^2 y)+d\cdot (1)\\
&=a \cdot (x^2)+ b \cdot (y) +c \cdot (x^2 y)+d \cdot (1),\\
\widehat{B}&=e\cdot ((x^2y)^{-1} x^{-1}y)+ f \cdot ((x^2 y)^{-1} )+ g\cdot (x^{-1} y)+ h\cdot (1)\\
&=e \cdot (x^3) + f\cdot (x^2 y) + g\cdot (x^3y)+ h\cdot (1).
\end{align*}
We compute the product $\widehat{A} \cdot \widehat{B}$ in $\mathbb{C}\left[ D_4\right]$,
\begin{multline*}
\widehat{A}\cdot\widehat{B}=(ae+bg)\cdot (x)+ (af+bh) \cdot (y) + (ce+dg)\cdot (x^3y)+ (cf+dh)\cdot (1)\\
 +(ag+be) \cdot (xy) + (ah+bf) \cdot (x^2) + (cg+de)\cdot (x^3)+(ch+df) \cdot (x^2y)
\end{multline*}
and observe that the first four terms and last four terms are precisely the entries of
\[
M_1\coloneqq\begin{bmatrix}
a & b\\
c & d
\end{bmatrix} \begin{bmatrix}
e & f\\
g & h
\end{bmatrix}
\qquad \text{and} \qquad
M_2\coloneqq\begin{bmatrix}
a & b\\
c & d
\end{bmatrix} \begin{bmatrix}
g & h\\
e & f
\end{bmatrix}
\]
respectively.
In other words, we can calculate $M_1$ and $M_2$ simultaneously by calculating $\widehat{A} \cdot \widehat{B}$.
On the other hand, $D_4$ has four irreducible representations of dimension one and one of dimension two:
\begin{enumerate}[\upshape (i)]
\item trivial: $(1,x,x^2,x^3,y,xy,x^2y,x^3y)\overset{\rho}{\mapsto} (1,1,1,1,1,1,1,1)$
\item sign type $1$: $(1,x,x^2,x^3,y,xy,x^2y,x^3y)\overset{\rho}{\mapsto} (1,1,1,1,-1,-1,-1,-1)$
\item sign type $2$: $(1,x,x^2,x^3,y,xy,x^2y,x^3y)\overset{\rho}{\mapsto} (1,-1,1,-1,1,-1,1,-1)$
\item sign type $3$: $(1,x,x^2,x^3,y,xy,x^2y,x^3y)\overset{\rho}{\mapsto} (1,-1,1,-1,-1,1,-1,1)$
\item two-dimensional:
\[ 
\begin{aligned}
1&\mapsto \begin{bmatrix}
1 & 0\\
0 & 1
\end{bmatrix}, \quad &
x&\mapsto \begin{bmatrix}
0 & -1\\
1 & 0
\end{bmatrix}, \quad &
x^2&\mapsto \begin{bmatrix}
-1 & 0\\
0 & -1
\end{bmatrix}, \quad &
x^3&\mapsto \begin{bmatrix}
0 & 1\\
-1 & 0
\end{bmatrix},\\
y&\mapsto \begin{bmatrix}
1 & 0\\
0 & -1
\end{bmatrix}, \quad &
xy&\mapsto \begin{bmatrix}
0 & 1\\
1 & 0
\end{bmatrix}, \quad &
x^2y&\mapsto \begin{bmatrix}
-1 & 0\\
0 & 1
\end{bmatrix}, \quad &
x^3y&\mapsto \begin{bmatrix}
0 & -1\\
-1 & 0
\end{bmatrix}.
\end{aligned}
\]
\end{enumerate}
By Wedderburn Theorem we have
\[
\mathbb{C}\left[D_4\right]\simeq \mathbb{C}\oplus \mathbb{C}\oplus \mathbb{C}\oplus \mathbb{C}\oplus \mathbb{C}^2\otimes \mathbb{C}^2,
\]
where the first four $\mathbb{C}$'s correspond to the four $1$-dimensional representations and the $\mathbb{C}^2\otimes \mathbb{C}^2$ corresponds to the $2$-dimensional representation. Under this isomorphism, we may identify $\widehat{A}$ and $\widehat{B}$ as $6\times 6$ block diagonal matrices,
{\footnotesize
\[
\widehat{A}=\begin{bmatrix}
a+b+c+d & 0 &0 &0 &0 &0\\
0 & a-b-c+d & 0 &0 &0 &0 \\
0 & 0 & a+b+c+d & 0 & 0 & 0\\
0 &0 &0 & a-b-c+d & 0& 0\\
0 &0 & 0 &0 &-a+b-c+d & 0\\
0  &0 &0 &0 &0 & -a-b+c+d   
\end{bmatrix}
\]}%
and
{\footnotesize
\[
\widehat{B}=\begin{bmatrix}
e+f+g+h & 0 &0 &0 &0 &0\\
0 & e-f-g+h & 0 &0 &0 &0 \\
0 & 0 & -e+f-g+h & 0 & 0 & 0\\
0 &0 &0 & -e-f+g+h & 0& 0\\
0 &0 & 0 &0 & -f+h & e-g\\
0  &0 &0 &0 &-e-g & f+h    
\end{bmatrix}.
\]}%
Hence the computation of $\widehat{A}\cdot \widehat{B}$ costs eight multiplications.
\end{proof}
\begin{corollary}
Suppose
\[
\begin{bmatrix}
a & b\\
c & d
\end{bmatrix} \begin{bmatrix}
g & h\\
e & f
\end{bmatrix}=0.
\]
Then the product
\[
\begin{bmatrix}
a & b\\
c & d
\end{bmatrix} \begin{bmatrix}
e & f\\
g & h
\end{bmatrix}
\]
can be computed with four multiplications. 
\end{corollary}
\begin{proof}
If the given condition holds, one may obtain the required product from the first four diagonal entries of $\widehat{A} \cdot \widehat{B}$, which costs four multiplications.
\end{proof}

We restate Proposition~\ref{prop:simultaneous matrix multiplication} in terms of the structure tensor $\mu_\mathsf{f}$ of the bilinear map 
\[
\beta_\mathsf{f}: \mathbb{C}^{2 \times 2} \times \mathbb{C}^{2 \times 2} \to \mathbb{C}^{2 \times 2}\oplus\mathbb{C}^{2 \times 2}, \quad (A,B)\mapsto (AB,AB^{\mathsf{f}}),
\]
where $B^{\mathsf{f}}$ denotes the operation of switching the first and second row of $B$. Note that
\[
\mu_\mathsf{f} \in (\mathbb{C}^{2 \times 2})^*\otimes (\mathbb{C}^{2 \times 2})^*\otimes (\mathbb{C}^{2 \times 2}\oplus \mathbb{C}^{2 \times 2})\simeq \mathbb{C}^4\otimes \mathbb{C}^4\otimes \mathbb{C}^8.
\]
\begin{proposition}\label{prop:rankf}
The rank and border rank of the structure tensor for the simultaneous matrix-matrix product in \eqref{eq:sim1} are given by
\[
\rank(\mu_\mathsf{f})=\brank(\mu_\mathsf{f})=8.
\]
\end{proposition}
\begin{proof}
It is easy to verify that $\operatorname{span}(\mu_\mathsf{f}(\mathbb{C}^{2 \times 2}\otimes \mathbb{C}^{2 \times 2}))=\mathbb{C}^{2 \times 2}$. Hence the required result follows from Propositions~\ref{prop:rank lower bound}, \ref{prop:border rank equals rank}, and \ref{prop:simultaneous matrix multiplication}.
\end{proof}
\begin{corollary}\label{cor:abbf}
Consider the matrices
\begin{equation}\label{eq:abbf}
A=\begin{bmatrix}
a & b\\
c & d
\end{bmatrix} \in \mathbb{C}^{2 \times 2}, \quad B= \begin{bmatrix}
e_1 & e_2 & \cdots & e_{2n}\\
f_1 & f_2 & \cdots & f_{2n} \end{bmatrix} \in \change{\mathbb{C}^{2 \times 2n}}, \quad
B^\mathsf{f}
= \begin{bmatrix}
f_1 & f_2 & \cdots & f_{2n}\\
e_1 & e_2 & \cdots & e_{2n} \end{bmatrix} \in \change{\mathbb{C}^{2 \times 2n}},
\end{equation}
where $n$ is any positive integer. Then $AB$ and $AB^{\mathsf{f}}$ can be computed simultaneously with $8n$ multiplications.
\end{corollary}
\begin{proof}
We may realize the bilinear map
\[
\mathbb{C}^{2 \times 2}\times \mathbb{C}^{2\times 2n}\to \mathbb{C}^{2\times 2n}\oplus \mathbb{C}^{2\times 2n}, \quad
(A,B)\mapsto (AB,AB^{\mathsf{f}})
\]
by the algebra $\mathbb{C}[D_4]\times \cdots \times \mathbb{C}[D_4]$ ($n$ copies).
\end{proof}

Suppose we are instead interested in computing 
\[
\begin{bmatrix}
a & b \\
c & d
\end{bmatrix}\begin{bmatrix}
e & f \\
g & h
\end{bmatrix}\qquad \text{and} \qquad\begin{bmatrix}
a & b \\
c & d
\end{bmatrix}\begin{bmatrix}
h & g \\
e & f
\end{bmatrix}
\]
simultaneously. We start by realizing $2 \times 2$ matrix product by the algebra $\mathbb{C}[x]/(x^8-1)$. Let $A,B \in \mathbb{C}^{2 \times 2}$ be as in \eqref{eq:ab}. Consider the embedding
\begin{align*}
j : \mathbb{C}^{2 \times 2}\otimes \mathbb{C}^{2 \times 2} &\to \mathbb{C}[x]/(x^8-1)\otimes \mathbb{C}[x]/(x^8-1),\\
(A,B) &\mapsto (ax^3 + cx^2 + bx + d, gx^6 + ex^4 + hx^2 + f),
\end{align*}
and the projection
\[
\proj:\mathbb{C}[x]/(x^8-1)\to \mathbb{C}^{2 \times 2}, \qquad
\sum_{i=1}^7 u_i x^i \mapsto\begin{bmatrix}
u_7 & u_3\\
u_6 & u_2
\end{bmatrix}.
\]
We may verify that for these choices, the diagram in \eqref{eq:gentpp0} commutes. The product
\[
(ax^3 + cx^2 + bx + d)(gx^6 + ex^4 + hx^2 + f) 
\]
in $\mathbb{C}[x]/(x^8-1)$ gives us the following counterpart of Proposition~\ref{prop:simultaneous matrix multiplication}.
\begin{proposition}\label{prop:simultaneous multiplication 2}
The following two matrix-matrix products:
\begin{equation}\label{eq:sim2}
\begin{bmatrix}
a & b\\
c & d
\end{bmatrix} \begin{bmatrix}
e & f\\
g & h
\end{bmatrix}
\qquad \text{and} \qquad \begin{bmatrix}
a & b \\
c & d
\end{bmatrix}\begin{bmatrix}
h & g \\
e & f
\end{bmatrix}
\end{equation}
can be computed simultaneously with eight multiplications.
\end{proposition}
Again, we may restate Proposition~\ref{prop:simultaneous multiplication 2} in terms of the structure tensor $\mu_\mathsf{g}$ of the bilinear map 
\[
\beta_\mathsf{g}:\mathbb{C}^{2 \times 2}\times \mathbb{C}^{2 \times 2} \to \mathbb{C}^{2 \times 2}\oplus \mathbb{C}^{2 \times 2},\quad
(A,B)\mapsto (AB,AB^{\mathsf{g}})
\]
where $B^{\mathsf{g}}$ is the matrix obtained from $B$ by switching the first row and the second row and then switching the first and the second entry in the first row. The following analogue of Proposition~\ref{prop:rankf} follows from Propositions~\ref{prop:rank lower bound}, \ref{prop:border rank equals rank}, and \ref{prop:simultaneous multiplication 2}.
\begin{proposition}
The rank and border rank of the structure tensor for the simultaneous matrix-matrix product in \eqref{eq:sim2} are given by
\[
\rank(\mu_\mathsf{g})=\brank(\mu_\mathsf{g})=8.
\] 
\end{proposition}

We also have the following analogue of Corollary~\ref{cor:abbf}.
\begin{corollary}
Let $A \in \mathbb{C}^{2 \times 2}$  and $B\in \mathbb{C}^{2 \times 2n}$ be as in \eqref{eq:abbf}.
Let $B^{\mathsf{g}} \in  \mathbb{C}^{2 \times 2n}$ be the matrix obtained from $B$ by switching the first and second row followed by switching $2i$th and $(2i-1)$th entry in the first row for $i=1,2,\dots,\lfloor n/2\rfloor$. Then $AB$ and $AB^{\mathsf{g}}$ can be computed simultaneously with $8n$ multiplications. 
\end{corollary}
\begin{proof}
We may realize the bilinear map 
\[
\mathbb{C}^{2 \times 2}\times \mathbb{C}^{2\times 2n}\to \mathbb{C}^{2\times 2n} \oplus \mathbb{C}^{2\times 2n}, \quad
(A,B)\mapsto (AB, AB^{\mathsf{g}})
\]
by the algebra $\mathbb{C}[x_1,\dots, x_n]/(x_i^8-1\mid i=1,\dots,n)$.
\end{proof}

\section{Conclusion}

The Strassen tensor rank approach gives us a simple way for quantifying bilinear complexity whereas the (generalized) Cohn--Umans approach gives us a constructive way that allows for the rich properties of various algebras to be used in analyzing bilinear complexity. The two methods can be applied hand-in-hand to systematically discover algorithms of optimal bilinear complexity.

\subsection*{Acknowledgment}

We thank Henry Cohn for very helpful discussions that initiated this work. We are also grateful to Andrew Chien, Nikos Pitsianis, and Xiaobai Sun for answering our questions about energy costs and circuit complexity of various integer and floating point operations; to Mike Stein for suggesting that we examine \textsc{bttb} matrices; and to Chris Umans for prompting \change{Construction~\ref{construction}. We thank the two anonymous referees and the handling editor for their exceptionally helpful comments and constructive suggestions. In particular, we included Sections~\ref{sec:overview} and \ref{sec:stability} at the handling editor's urging, which in retrospect were glaring omissions.}

LHL and KY are partially supported by AFOSR FA9550-13-1-0133, DARPA D15AP00109, NSF IIS 1546413, DMS 1209136, and DMS 1057064. In addition, KY's work is also partially supported by NSF CCF 1017760.

\end{document}